%% file: main.tex
\documentclass[11pt,letterpaper]{amsart}

\usepackage{amsthm}
\usepackage{amsmath}
\usepackage{amssymb}
\usepackage{enumerate}
\usepackage{hyperref}
\usepackage{mathrsfs}
\usepackage{color}
\usepackage{bm}
\usepackage{todonotes}
\usepackage{ esint }
\usepackage{caption}
\usepackage{subcaption}
\usepackage{algpseudocodex}

\title[Pseudospectral approach to resonances]{A pseudospectral approach to rigorous numerical estimation of resonances of transfer operators}
\subjclass{Primary: 37C30, 47A10; Secondary: 65G99.}
\keywords{Ruelle-Pollicott resonances; transfer operators; pseudospectrum; rigorous numerics; computer-assisted proof; spectral approximation; uniformly expanding maps.}
\author{Alex Blumenthal}
\address{School of Mathematics, Georgia Institute of Technology, Atlanta, GA 30332, USA}
\email{ablumenthal6@gatech.edu}

\author{Isaia Nisoli}
\address{Instituto de Matem\'{a}tica, Universidade Federal do Rio de Janeiro, Rio de Janeiro, Brazil}
\email{nisoli@im.ufrj.br}

\author{Toby Taylor-Crush}
\address{Department of Mathematical Sciences, Loughborough University, Loughborough LE11 3TU, United Kingdom}

\thanks{A.B.
was supported by National Science Foundation grants DMS-2009431 and DMS-2237360. I. N. was partially supported by CAPES-Finance Code 001, CNPq Projeto Universal No. 404943/2023-3, 
CAPES-PRINT 88881.311616/2018-00, CAPES-STINT 88887.155746/2017-00.}

\input{macros.tex}

\begin{document}

\begin{abstract}
    {Ruelle-Pollicott} resonances, isolated eigenvalues of a transfer operator acting on suitably chosen Banach spaces, play a fundamental role in understanding the statistical properties of chaotic dynamical systems. In this paper, we introduce a pseudospectral approach, inspired by Householder's theorem, for the rigorous, computer-assisted estimation of resonances, providing regions where resonances must exist and precluding the presence of resonances elsewhere. The approach is general, and applies to the transfer operators of a wide variety of chaotic systems, including Anosov/ Axiom A diffeomorphisms and piecewise expanding maps. We implement this approach  computationally for a class of analytic uniformly expanding maps of the circle. We anticipate that the pseudospectral framework developed here will be broadly applicable to other spectral problems in dynamical systems and beyond. 
\end{abstract}

\maketitle

\noindent\textit{To appear in Foundations of Computational Mathematics.}\\
\textit{Communicated by Fran\c{c}oise Tisseur.}
\medskip

\section{Introduction}

\input{intro.tex}

\section{General method for spectral enclosure}\label{sec:householder}
\input{householder.tex}

\section{Analytic Framework} \label{sec:analyticframework}

\input{analytic.tex}

\input{examples.tex}

\input{validated.tex}

\section{Comparisons and Outlook}
\input{outlook.tex}

\bibliographystyle{alphaurl}
\bibliography{biblio}

\end{document}

%% file: macros.tex
\newtheorem{theorem}{Theorem}[section]
\newtheorem{corollary}[theorem]{Corollary}
\newtheorem{lemma}[theorem]{Lemma}
\newtheorem{proposition}[theorem]{Proposition}

\theoremstyle{definition}

\newtheorem{definition}[theorem]{Definition}
\newtheorem{remark}[theorem]{Remark}
\newtheorem{assumption}[theorem]{Assumption}
\newtheorem{claim}[theorem]{Claim}
\newtheorem{example}[theorem]{Example}

\newcommand{\C}{\mathbb{C}}

\newcommand{\R}{\mathbb{R}}

\newcommand{\Z}{\mathbb{Z}}

\newcommand{\Bc}{\mathcal{B}}

\newcommand{\Lc}{\mathcal{L}}

\newcommand{\Ac}{\mathcal{A}}

\newcommand{\Dc}{\mathcal{D}}

\newcommand{\e}{\epsilon}

\newcommand{\dist}{\operatorname{dist}}

\newcommand{\Leb}{\operatorname{Leb}}
\newcommand{\Lip}{\operatorname{Lip}}

\newcommand{\edited}[1]{{\color{black}#1}}

%% file: intro.tex
Let $T : M \to M$ be a mapping of a compact, connected Riemannian manifold $M$. Given a reference probability measure $\mu$ on $M$, e.g., Lebesgue measure $\Leb_M$, the \emph{transfer operator} $\Lc = \Lc_{T, \mu}$ of $T$ with respect to $\mu$ is given by
\[\Lc \varphi = \frac{d(T_* \nu)}{d \mu} \, , \]
for $\varphi \in L^1(\mu)$, where $\nu$ is the (possibly signed) measure for which $d \nu = \varphi d \mu$, and where $T_* \nu := \nu \circ T^{-1}$ is the \emph{pushforward} of $\nu$. The transfer operator $\Lc_{T, \mu}$ is defined for all $\varphi \in L^1(\mu)$ so long as $T$ satisfies the nondegeneracy condition $T_* \mu \ll \mu$. This holds for $\mu = \Leb_M$ when, e.g., $T$ is a local diffeomorphism.

The transfer operator has the following important dynamical implication: if $X$ is a random initial condition distributed like some Borel probability $\nu$ on $M$, then $T_* \nu$ is the law of \edited{$T(X)$}. The transfer operator $\Lc$ is the same, except now we track densities with respect to the reference measure $\mu$.
As such, transfer operators are powerful tools in the analysis of statistical properties of dynamical systems. For instance: a nonnegative eigenfunction $\varphi \in L^1(\mu)$ with eigenvalue 1 gives rise to a $T$-invariant probability measure $d \nu(x) = (\int \varphi d \mu)^{-1} \varphi(x) d \mu(x)$\edited{,} while spectral properties of $\Lc$ on appropriate function spaces dictate rates at which the invariant measure $\nu$ emerges as a system \edited{equilibrates} under repeated iterations of $T$.  

The simplest setting, and that of primary interest in this manuscript, is the case when $T$ is a smooth, uniformly expanding mapping of $M$. Here, it is natural to take  $\mu = \Leb_M$, the normalized Riemannian volume on $M$, and to restrict $\Lc$ to act on an appropriate ``positive-regularity'' subspace of $L^1(\mu)$, e.g., $\Bc = \Lip$, the set of Lipschitz-continuous functions. 
In this setting, it is well-known that 
\begin{itemize}
	\item[(i)] $\Lc$ admits a simple eigenvalue at 1;
	\item[(ii)] the essential spectrum $\sigma_{\rm ess}(\Lc)$ is contained in $B_r(0)$, the open ball \edited{of} radius $r$ centered at $0$, for some $r < 1$; and
	\item[(iii)] $\sigma(\Lc) \setminus \{ 1\}$ is contained inside the open unit disk $B_1(0)$.
\end{itemize}

Item (i) implies unique existence of an invariant density $\varphi_T$ for $T$. This measure is sometimes called \emph{physical}, in that it governs the asymptotic behavior of a positive Lebesgue-measure set of initial conditions \cite{eckmann1985ergodic}.
Item (ii) implies that $\sigma(\Lc) \setminus B_r(0)$ consists of pure point spectrum, and together with (iii) yields quantitative estimates on the rate at which this system relaxes to the invariant density $\varphi_T$: to wit, if $X$ is a random initial condition with density $\varphi_0$, then $T^n(X)$ has density \edited{$\varphi_n = \Lc^n \varphi_0$} for which
\[\varphi_n = \varphi_T + \sum_{i = 1}^k c_i \psi_i \lambda_i^n + O(r^n) \, ,  \]
where $\lambda_1, \dots, \lambda_k$ are the eigenvalues of $\Lc$ in $B_1(0) \setminus B_r(0)$ with corresponding eigenfunctions $\psi_1, \dots, \psi_k$, the $c_i = c_i(\varphi_0)$ are complex coefficients, and $O(r^n)$ is an error term of size $\leq C r^n$ in the norm on $\Bc$. 
Thus, the point spectrum $\{ \lambda_i\}$, known in the literature as \emph{Ruelle-Pollicott resonances}, captures the longest-lived transient behavior of \edited{$T$} and dictates a sharp asymptotic rate of correlation decay.
For further discussion of this and related results, we refer the reader to the books \cite{baladi2000positive} or \cite{demers2021transfer}.

The \edited{foregoing} abstract theory is versatile and admits a variety of extensions, e.g. to piecewise expanding and uniformly hyperbolic $T$, up to a suitable choice of $\Bc$. However, outside special classes of $T$, it is regarded as a very difficult problem to prove, for a \emph{given} mapping $T$, the existence of resonances and  estimates of their location. For further discussion, see Section \ref{subsec:priorWork} below.

\bigskip

In this paper, we put forward an approach for the \emph{rigorous, computer-assisted estimation} of spectrum of a wide class of bounded linear operators on infinite dimensional spaces, amenable to the transfer operator setting presented above. 
The following is a sample result of our method. Here, we parametrize the circle $M = S^1$ by the unit circle in the complex plane $\mathbb{C}$, and consider the mapping 
\[ T(z) = i z^2 \exp \left[{\left(\frac12 - \pi\left(\frac{5}{64}+\frac{1}{128}+\frac{1}{256}\right)\right)(z - z^{-1})}\right] \] 
with reference measure $\mu = \Leb_{S^1}$. Below, we regard the operator $\Lc_T$ as acting on a suitable space of functions $\mathcal{A}$ on $S^1$ admitting holomorphic extension to a neighborhood of the unit circle; see Section \ref{sec:analyticframework} for details. For comments on our particular choice of $T$, see Remark \ref{rmk:choiceOfT4}. 

\begin{theorem}[Informal]\label{thm:exampleIntro}
	The transfer operator $\Lc_T$ gives rise to a bounded linear operator $\Ac \to \Ac$, and its spectrum $\sigma(\Ac)$ is contained in the union of the disks $F_0, F_1, F_2, F_3$ depicted in Figure \ref{fig:circle_enc_Arnold_intro}. 

\begin{figure}
     \centering
     \begin{subfigure}{0.45\textwidth}
         \centering
         \includegraphics[width=\linewidth]{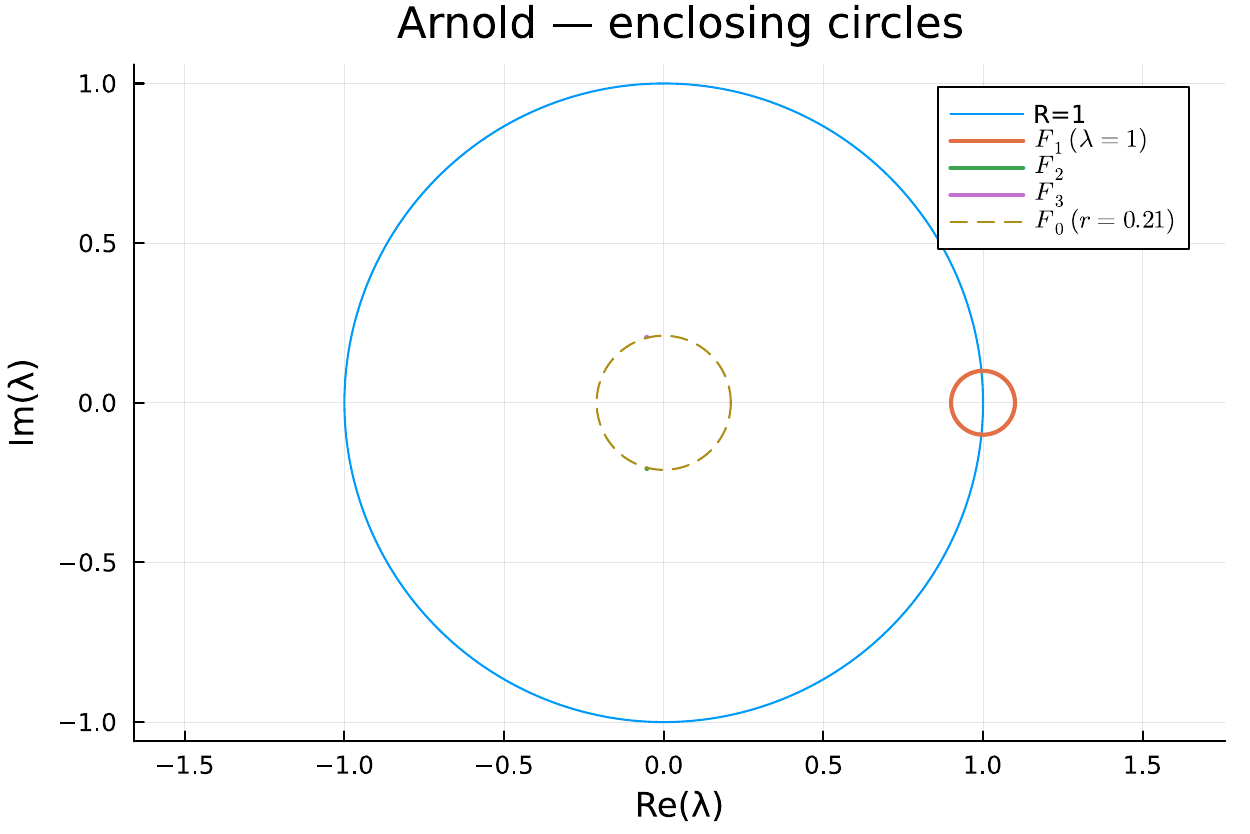}
         \caption{}
         \label{fig:circles_Arnold_intro_a}
     \end{subfigure}
     \begin{subfigure}{0.45\textwidth}
         \centering
         \includegraphics[width=\linewidth]{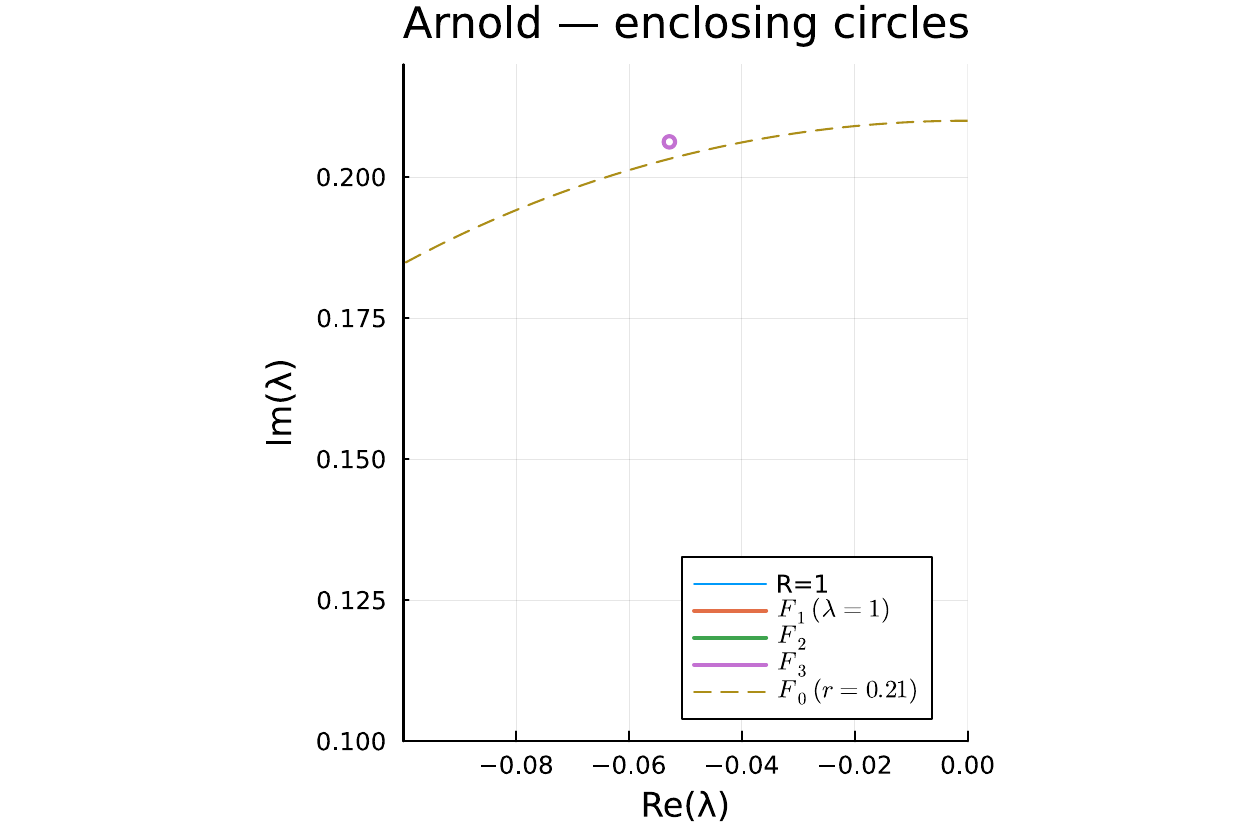}
         \caption{}
         \label{fig:circles_Arnold_intro_b}
     \end{subfigure}
     \caption{Circle enclosure for the perturbed doubling map (Theorem \ref{thm:exampleIntro}). Figure \ref{fig:circles_Arnold_intro_a} depicts the circles $F_0, F_1$ in Theorem \ref{thm:exampleIntro} enclosing $0$ and $1$, respectively. Figure \ref{fig:circles_Arnold_intro_b} is zoomed-in to depict the circle $F_2$ enclosing resonance $\lambda$, which is small and situated close to $F_0$. Circle $F_3$ enclosing resonance $\bar \lambda$ is similar. }
     \label{fig:circle_enc_Arnold_intro}
\end{figure}
  
	Moreover, the multiplicity of $\sigma(\Lc_T)$ is equal to 1 in each of the disks $F_1, F_2, F_3$. 
\end{theorem}
The above is re-stated precisely, with numerical values, as Theorem \ref{thm:pertDoubleMap4} below. 
Before moving on, we highlight that Theorem \ref{thm:exampleIntro} provides two distinct but related pieces of information: 
\begin{itemize}
	\item[(1)] the operator $\Lc_T : \Ac \to \Ac$ possesses simple resonances within each of the disks $F_1, F_2, F_3$; and 
	\item[(2)] there are \emph{no resonances} outside $\cup_0^3 F_i$. 
\end{itemize}
Given (1), numerical estimation of the associated eigenfunctions can be used to further refine the location of the resonance, e.g., using the Newton-Kantorovich method. Point (2), on the other hand, has \emph{ruled out the presence of spectrum in a given region}, a.k.a. \emph{spectral exclosure} -- a task which is far more challenging to accomplish with existing computer-assisted techniques for the estimation of spectrum. 

\bigskip

We now turn to a brief summary of the general method. This is followed in Section \ref{subsec:priorWork} with some background on previous works on transfer operators as well as on general methods for estimation of spectrum of operators. We close in Section \ref{subsec:plan1} with a brief agenda for the rest of the paper.

\subsection{An abstract pseudospectral method for enclosure of point spectrum}\label{subsec:abstractHouseholderIntro}

At a broad level, ours is an adaptation of the classical Householder argument for the stability of spectrum under perturbation. 
Below, for an $n \times n$ matrix $Z$ with complex entries, the $\delta$-\emph{pseudospectrum} $\sigma_\delta(Z)$ is defined by
\[\sigma_\delta(Z) = \sigma(Z) \cup \{ z \in \C \setminus \sigma(Z) : \| (z - Z)^{-1} \| \geq \delta^{-1} \} \,, \]
where $\| \cdot \|$ refers to any induced matrix norm, e.g., that coming from the Euclidean norm. 
\begin{theorem}[\cite{bauer1960norms, householder1958approximate}]\label{thm:householderIntro}
	Let $A, \tilde{A}$ be $n \times n$ matrices and let $\delta > 0$. Then, 
	\begin{align}\label{eq:bauerFike}
		\sigma(A) \subset \sigma_\delta(\tilde{A}) \quad \text{ for all } \quad \delta \geq \| A - \tilde{A}\| \, ,
	\end{align}
\end{theorem}
In view of the proof's brevity and elegance, we summarize it here. 
\begin{proof}
Given $\lambda \in \sigma(A)$ and an associated eigenvector $v \in \C^n \setminus \{ 0 \}$, one has 
\[(A - \tilde{A}) v = (\lambda - \tilde{A}) v \, . \]
If $\lambda \in \sigma(\tilde{A})$, there is nothing to prove. If $\lambda \notin \sigma(\tilde{A})$, we have 
\[v = (\lambda - \tilde{A})^{-1} (A - \tilde{A}) v \, ,\]
which implies 
\[\| (\lambda - \tilde{A})^{-1}\| \geq \| A - \tilde{A}\|^{-1} \, \qedhere \]
\end{proof}

Theorem \ref{thm:householderIntro} is \edited{Householder-type spectral inclusion} \cite{bauer1960norms, householder1958approximate}; see, e.g., the excellent discussion in \cite[Section I.2]{trefethenEmbree}. 
For our purposes, we are interested in the interpretation of \eqref{eq:bauerFike} as an \emph{enclosure} of the spectrum of $A$ by the pseudospectrum of $\tilde{A}$. This is particularly relevant to computer-assisted proof when $A$ is regarded as a ``platonic'' or ``idealized'' matrix and $\tilde{A}$ is an approximation of $A$ with floating-point entries, with the hope that $\sigma_\delta(\tilde{A})$ can be estimated using a computer. If $\| A - \tilde A\|$ is very small, then we are permitted to take yet smaller values of $\delta$; as $\delta \to 0$, the enclosing set $\sigma_\delta(\tilde A)$ becomes a tighter enclosure of the eigenvalues of $\tilde A$, hence those of $\tilde A$ converge to the eigenvalues of the `platonic' operator $A$. \edited{Indeed, $\sigma_\delta(\tilde{A})$ is ``tight'' as an approximation for $\sigma(A)$:  
\begin{theorem}[{\cite[Theorem 2.1]{trefethen1999pseudospectra}}]\label{thm:modulusCtyPseudospec1}
	Let $B \in \mathcal{M}_{n \times n}(\mathbb{C})$ and let $\delta > 0$. Then,
	\[\sigma_\delta(B) = \{ z \in \sigma(B + E) \text{ for some } E \in \mathcal{M}_{n \times n}(\C) \text{ such that } \| E \| \leq \delta\}\, . \]
\end{theorem}
}

\begin{figure}
\begin{center}
\includegraphics[width=70mm]{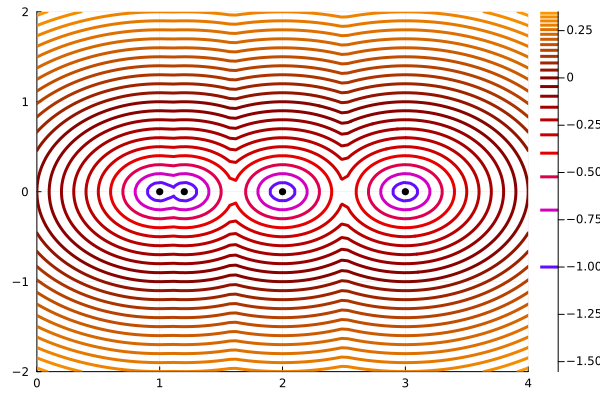}
\caption{Numerical plot (non-certified) of the pseudospectrum levels with respect to the spectral norm of a perturbation $\tilde{A}$ of the diagonal matrix $A$ with entries $(1.0, 1.2, 2.0, 3.0)$, such that 
$||A-\tilde{A}||_{\infty}\leq 0.0001$. The contour curves correspond to levels of $-\log_{10}(||(z-\tilde{A})^{-1}||_2)$.  We refer to \cite{trefethenEmbree} for the numerical methods used to compute this graph.
Note how the pseudospectrum behaves between $1.0$ and $1.2$.}		
\end{center}
\end{figure}

We emphasize that \edited{versions of} Theorems \ref{thm:householderIntro} \edited{and \ref{thm:modulusCtyPseudospec1} hold}, when $A, \tilde{A}$ are operators on a possibly infinite-dimensional Banach space $\Bc$ \edited{(see Section I.4 of \cite{trefethenEmbree})}. On the other hand, in computer-assisted proof one is usually forced to take the `approximate' operator $\tilde{A}$ to be finite-rank, so as to be manipulated with floating-point arithmetic in an array of finite size. As such, it is natural to work with a sequence $\Pi_n$ of finite-rank projections on $\Bc$, and to take $\tilde{A} = \tilde{A}_n = \Pi_n \circ A \circ \Pi_n$, with the hope that increasing $n$ will reduce error in the approximation. 
For the applications we have in mind, however, $\Bc$ is infinite-dimensional, and \emph{it is impossible for $\Pi_n$ to converge to the identity operator} $\operatorname{Id}$ in the norm on $\Bc$. Absent special structure in the operator $A$, then, there is little hope that $\tilde{A}_n \to A$ in operator norm. 

Typically, however, one has $\Pi_n \to \operatorname{Id}$ in other operator topologies, e.g., the operator norm with a mismatch between the norm assigned to the domain and codomain, as in the following illustrative example with a long history in the finite-dimensional approximation of transfer operators \cite{ulam1960collection, li1976finite}.

\begin{example}\label{exam:BVapprox}
	Let $\Bc = BV([0,1])$, the space of functions of \emph{bounded variation} with norm $|\varphi|_{BV} = |\varphi|_{L^1} + \operatorname{Var}(\varphi)$. Here $|\cdot|_{L^1}$ is the usual $L^1$ norm, and $\operatorname{Var}(\varphi)$ is the \emph{first variation}, which for smooth $\varphi$ coincides with the $W^{1,1}$ norm, i.e., $\operatorname{Var}(\varphi) = \int_0^1 |\varphi'| dx$. 
	
	For $\varphi \in \Bc$ and $n$ fixed, define $\Pi_n \varphi : [0,1] \to \R$ by 
	\[\Pi_n \varphi = \sum_i c_i(\varphi) {\bf 1}_{I_i} \, , \]
	where $I_i := [(i-1)/n, i/n)$, $c_i(\varphi) = n \int_{I_i} \varphi dx$ and ${\bf 1}_K$ is the indicator function of $K \subset [0,1]$. It is straightforward to check that $\Pi_n \varphi \to \varphi$ in $L^1$ for all $\varphi \in L^1$, but \edited{$\Pi_n \to \operatorname{Id}$} fails in the norm topology on $L^1$. At best, one has $| \Pi_n \varphi - \varphi|_{L^1} \leq \frac{1}{2n} |\varphi|_{BV}$ \cite{galatolo2020existence}. 
\end{example}

The core idea, presented in Section \ref{sec:householder} below, is that the Householder argument can be applied to discretization schemes like that presented in Example \ref{exam:BVapprox} under some appropriate conditions on the operator $A$ and the truncations $\Pi_n$. In this paper, the conditions we formulate are inspired by the framework of \emph{Doeblin-Fortet-Lasota-Yorke (DFLY) inequalities} and \emph{Hennion's Theorem}, a set of classical tools for controlling the essential spectral radius of transfer operators.

\subsection{Discussion of prior work}\label{subsec:priorWork}

Below, we discuss prior work on: spectral approximation for transfer operators (Section \ref{subsubsec:TferSpecApproxIntro}); \edited{relevant results on pseudospectrum} (Section \ref{subsubsec:genlSpecApproxIntro}); and computer-assisted estimation of spectrum in contexts beyond transfer operators (Section \ref{subsubsec:CAPintro}). Additional discussion, including a comparison of our methods with some of those listed below, is given in Section \ref{sec:outlook}.

\subsubsection{Spectral approximation of transfer operators}\label{subsubsec:TferSpecApproxIntro}

The method we present in this paper for estimating the discrete spectrum of a transfer operator is far from the only one -- over the last three decades there have been numerous works outlining a variety of approaches one could take for estimating the spectrum of a transfer operator in terms of various finite-rank approximations. 

Several works take advantage of specific behavior of expanding and uniformly hyperbolic maps, e.g., the \emph{Markov property}, to build specially-tailored finite-rank approximations: see, e.g., the representative works \cite{froyland1997computer, froyland2007ulam} and the book chapters \cite[Chapters 6-7]{ding2010statistical}. We highlight \cite{froyland1997computer} in particular, where for piecewise expanding maps a discretization scheme is demonstrated to converge in a single norm topology, c.f. the discussion in Section \ref{subsec:abstractHouseholderIntro} above. This state of affairs is not typical of `most' discretization schemes, however, and relies crucially on the Markov property of the mapping -- something that the approach in this manuscript does not explicitly use. 

Other discretization schemes are possible, e.g., Lagrange interpolation in the case of holomorphic dynamics \cite{bandtlow2020lagrange}, mollifying the transfer operator using external stochastic forcing \cite{baladi1993spectra, dellnitz1999approximation}, or Galerkin truncation \cite{froyland2014detecting,wormell2019spectral}. Our own discretization scheme for analytic expanding maps, presented in Section \ref{sec:analyticframework}, is of Galerkin-type. An entirely different approach is possible via cycle expansions of periodic orbits-- see, e.g., \cite{artuso1990recycling,artuso1990recycling2,pollicott2000computing} {or the recent book chapter \cite{jenkinson2024dynamical}.} 

In rare cases, resonances can be exactly solved, as is the case for Blaschke products \cite{slipantschuk2013analytic, slipantschuk2017complete, slipantschuk2022resonances} or linear pseudo-Anosov flows \cite{faure2019ruelle}, but this situation is atypical, and usually the best one can hope for is some sort of approximation. 

A notable work for us is a set of results on the spectral stability of transfer operators satisfying a DFLY inequality due Keller and Liverani \cite{keller1999stability} (see also \cite{dyatlov2015stochastic} for analogous results in the case of Anosov flows). This too can be made into an approach to rigorously estimate resonances, as described in the overview paper \cite{liverani2001rigorous}. Similar spectral stability results in this setting are given in the contemporaneous paper \cite{baladi1999approximation}. Ultimately, our own approach for rigorous computer-assisted estimation of resonances of transfer operators will use the abstract results of \cite{keller1999stability}, as we describe in Section \ref{subsec:householderApp2}. 

We also acknowledge the substantial literature concerning rigorous com\-puter-as\-sisted estimation of the dominant eigenpair \cite{li1976finite, keller1982stochastic, pollicott2000computing, jenkinson2005orthonormal, galatolo2014elementary, wormell2019spectral, galatolo2023general} and related statistical properties, e.g., diffusion coefficients \cite{bahsoun2016rigorous, jenkinson2018rigorous}, large deviations rate functions and escape rates \cite{bahsoun2006rigorous, bahsoun2011invariant}, and linear response \cite{bahsoun2018rigorous, NisoliCrush2023}.

There are also a variety of pencil-and-paper works entraining location of \edited{non-dominant} spectrum. A natural idea in this vein is to attempt to control the spectral gap directly by establishing \edited{bounds on the mixing rate using positivity of} transfer operators-- see, e.g., \cite{rychlik1989regularity, liverani1995decayJSP, liverani1995decay, bressaud2002anosov, korepanov2019explicit}. We also note the paper \cite{butterley2021locating}, which uses a substantially different method to obtain finer information than just an estimate of the spectral gap.

In the physics literature, several methods have been put forward to (numerically, nonrigorously) estimate resonances from time series using the identification of resonances with poles of the correlation function. 
The paper \cite{florido2002schemes} documents such a framework using memory function methods with Pad\'e approximants and filter diagonalization along with some examples and citations to other papers -- see also the earlier paper \cite{isola1988resonances}. Other representative works in this direction include 
\cite{blum2000chaotic,sano2002parametric,venegeroles2008area, horvat2009hybrid}.
For a discussion of the rigorous identification of resonances with poles of the correlation function, see \cite[Chapter 1]{baladi2018dynamical}.

\edited{The approximation of spectral information of dynamical systems is also key goal of Koopmanism -- a collection of methodologies for extracting diagnostic and predictive information for possibly chaotic systems from data by approximating the \emph{Koopman operator} acting on observables. This is strongly related to the spectral theory of transfer operators, since the Koopman and transfer operators are dual to each other via the usual pairing of observables and measures. This is a broad and highly active subject. For now, we only mention here the representative papers \cite{budisic2012applied,brunton2022modern, das2020koopman, colbrook2024rigorous, williams2015data}, and those pertaining specifically to data-driven approximation of transfer operators \cite{froyland2009almost, froyland2014detecting, klus2018data}. We also highlight the recent paper \cite{akindji2026convergence}, which concerns the convergence of Extended Dynamic Mode Decomposition (EDMD) for a class analytic expanding interval maps similar to those in the present manuscript. 
We will return to a discussion of Koopmanism vis-\`a-vis pseudospectrum shortly. 
}

{Lastly, we highlight a distinctive the feature of our method: where possible, an emphasis on computer-assisted a posteriori estimates in lieu of abstract a priori bounds. In this vein, we were directly influenced by the earlier works \cite{galatolo2014elementary, galatolo2015elementary2, galatolo2020existence, galatolo2023general}. 
}

\subsubsection{Pseudospectrum and spectral approximation}\label{subsubsec:genlSpecApproxIntro}

\edited{

It is a classical result that the isolated point spectrum of linear operators varies continuously in the norm topology on operators. On the other hand, the rest of the spectrum, including approximate point and essential spectrum, varies at-best only upper-semicontinuously: these spectral components can `implode' under limits, even norm limits of bounded linear operators. 
This is tied to the phenomenon known as \emph{spectral pollution} \cite{davies2004spectral}, wherein spurious bands of spectra for approximations need not correspond to the spectrum of the limiting operator. This is a significant, overarching challenge in computational approximation of spectra across many fields. 

Pseudospectrum, on the other hand, turns out to enjoy better abstract stability properties under perturbation (Theorems \ref{thm:householderIntro}, \ref{thm:modulusCtyPseudospec1}), and forms the basis for many theoretical and practical approaches to spectral estimation. 

One notable application is to the study of computational complexity of spectrum estimation via the \emph{Solvability Complexity Index} (SCI). Roughly speaking, given a computational task (e.g., approximating the roots of an explicit function), the SCI is the minimal number of successive infinite limits an algorithm must take to converge to the true solution \cite{ben2015new}. For instance, approximating the spectrum of a finite-dimensional matrix has SCI 1, as is spectral approximation for a compact linear operator of a Hilbert space (due to the availability of finite-rank approximations). Generally speaking, one should only expect to practically realize SCI 1 tasks on a computer, while algorithms giving rise to CAP certificates are necessarily SCI 1. 

Notably, \cite{hansen2011solvability} provided an explicit pseudospectral algorithm demonstrating the (sharp) upper bound SCI $\leq 3$ for approximating spectrum of a general boudned linear operator of $\ell^2(\mathbb{N})$, with no further constraints. See \cite{colbrook2022foundations} for an expansive treatment of upper and lower bounds on SCI for a variety of constrained spectral approximation problems; see also \cite{ben2015new} and the more-recent PhD thesis \cite{colbrook2020foundations} for further treatments of this active subject.  

Additional works in and around these topics, not to do with SCI explicitly, include: pseudospectral estimation of approximate point spectrum \cite{wolff1998approximation, wolff2001discrete}; stability of the pseudospectrum itself under limits in the strong operator topology \cite{bogli2014remarks,bogli2018local}; and approximation of pseudospectrum under the strong operator topology via \emph{numerical ranges} of resolvents \cite{frommer2021pseudospectrum}.

}

\edited{Pseudospectrum also figures prominently in data-driven approaches for Koopman and transfer operators. The works \cite{colbrook2019romanhansen,colbrook2024townsend} introduce the ResDMD algorithm, a pseudospectral approach to dynamic mode decomposition, and \cite{boulle2025colbrookconradie} extends this to Koopman and transfer operators on reproducing kernel Hilbert spaces. 
ResDMD provides rigorous asymptotic convergence to the spectra and pseudospectra entirely from trajectory data -- well-suited to applications where the underlying system is not known explicitly. 
However, unlike the present manuscript, this algorithm stops short of CAP-certified spectral enclosure for any given finite computation, which requires an explicit description of the underlying dynamics to certify precise sup-resolvent bounds.  

We also highlight \cite{herwig2025slipantschuk}, which proposes a pseudospectral method to quantify errors in various DMD methods for transfer operators, using Blaschke maps on the circle -- similar to our setting of analytic expanding maps of the circle -- as an explicit testbed.



}



Another line of work \cite{bandtlow2008resolvent, sarihan2021quantitative} has to do with quantitative continuity of spectrum for classes of compact linear operators for which one has some a priori decay rate of eigenvalues, e.g., trace-class or exponential classes of operators. Both works are adaptations to the infinite-dimensional setting of an approach due to Henrici \cite{henrici1962bounds} relating resolvent estimates to spectral distances. See \cite{sarihan2021quantitative} for further work along these lines.

\subsubsection{Computer-assisted proof of estimation of spectrum}\label{subsubsec:CAPintro}

	There is a significant and growing literature of computer-assisted proof across dynamics and partial differential equations since the famous computer-assisted proof of the Feigenbaum conjecture by Lanford \cite{lanford1982computer}. 
	Beyond the works discussed in Section \ref{subsubsec:TferSpecApproxIntro}, there are many works in the dynamics of PDE and other infinite-dimensional dynamical systems addressing eigenvalue problems for linear operators, usually as a step on the way to proving something else. Recent examples include, e.g., assessing linear stability or instability of a stationary point or periodic orbit \cite{de2016connecting, james2017fourier,arioli2010integration} or computing Morse or Maslov indices \cite{arioli2019non, martine2022microscopic, van2024toward}. 
	Generally speaking, it is easier to estimate spectrum of self-adjoint operators, taking advantage of extra structure in that setting. For overview, see \cite[Chapter 10]{nakao2019numerical}. See also \cite[Section 9.3]{nakao2019numerical}, which treats non-selfadjoint spectral approximation. 
	
	A consistent theme in the estimation of spectra of linear operators is to treat the following two problems separately: 
	\begin{itemize}
		\item[(1)] Given a numerical eigenvalue, find nearby exact one. For this, a standard approach is to use Newton-Kantorovich iteration \cite[Section 9.3.1]{nakao2019numerical}. 
		\item[(2)] Once all likely eigenvalues have been found, rule out spectrum elsewhere -- what \cite{nakao2019numerical} refers to as \emph{spectral exclosure}. 
	\end{itemize}
	The latter problem is typically harder to treat, even in the selfadjoint case, due to spectral pollution. This is part of the appeal of the method in this paper, which treats the exclosure problem directly by confining spectrum of an operator $A$ to the pseudospectrum of an appropriate approximation $\tilde{A}$.

\bigskip

Our own approaches for estimating pseudospectrum with computer-as\-sisted proof are indebted to previous works in the area for finite matrices -- see, e.g., the books \cite{trefethen1999pseudospectra, varga2004circles}. We also single out the paper \cite{lui1997continuation}, which gives an iterative algorithms that converges on the pseudo-spectrum, relying on the continuity of $\|(zI-A)^{-1}\|$ in $z$, as well as the papers \cite{rump2010verified,miyajima2014verified} concerning the verification of bounds on the singular values of a matrix.

\subsection{Plan for the paper}\label{subsec:plan1}

{We begin in Section \ref{sec:householder} with a description of our general method for the approximation of spectrum using pseudospectrum, Proposition \ref{prop:householder}, culminating in Theorem \ref{thm:houseHolderSummary2} concerning pseudospectral approximations for operators satisfying a DFLY inequality. Section \ref{sec:analyticframework} specializes to the case of transfer operators of uniformly expanding maps on the circle (Theorem \ref{prop:householderForExpand3}). In Section \ref{sec:examples} we lay out two example mappings for which we enclose spectrum -- one of them the mapping presented earlier in Theorem \ref{thm:exampleIntro}, and the other a Blaschke product for which spectrum is exactly known \cite{bandtlow2017spectral}, which serves as a benchmark for the method. In Section \ref{sec:CAP5} we summarize some ideas appearing in the technical implementation of the computer-assisted proof. Finally, in Section \ref{sec:outlook} we offer some discussion on the general problem of spectral approximation.

\subsubsection*{Acknowledgements}

\edited{The authors would like to thank the following people for helpful discussions and suggestions during the course of writing this paper: Oscar Bandtlow, 
Maxime Breden, 
Matthew Colbrook, 
Stefano Galatolo, 
Jason Mireles-James, 
Federico Poloni, and
Julia Slipantschuk. }

We would also like to thank Stefano Galatolo and the Math Department at the University of Pisa, as well as Yuzuru Sato and RIES-Hokkaido University, for the use of computational resources. \edited{Lastly, we extend our gratidude to the anonymous referees of this paper, who pointed out errors and suggested key improvements to this manuscript. }

}

%% file: householder.tex

In Section \ref{subsec:householder} we state a general version of the Householder argument appropriate to our setting -- Proposition \ref{prop:householder} below. Section \ref{subsec:householderApp2} focuses on the application of these ideas to the control of spectrum of an operator in terms of that of an approximation (Theorem \ref{thm:houseHolderSummary2}). 
Finally, in Section \ref{subsec:diffSpaces} we address a subtle issue, relevant to transfer operators, concerning the ``feasibility'' of the setup in Section \ref{subsec:householderApp2} to applications. 

\subsection*{Notation}

Given a Banach space $(\Bc, |\cdot|)$ and a bounded linear $A : \Bc \to \Bc$, we will write $\sigma(A, \Bc)$ for the spectrum of \edited{$A$}, and $\sigma_{\rm point}(A, \Bc)$ for the isolated   point spectrum, i.e., the set of discrete eigenvalues of $A : \Bc \to \Bc$. We write $\rho(A, \Bc) = \C \setminus \sigma(A, \Bc)$ for the corresponding resolvent set. The resolvent operator itself will be written $R(z, A) := (z - A)^{-1}$ for $z \in \rho(A, \Bc)$, and $r_A(z) := |(z - A)^{-1}|$ for its operator norm. 
For a set $S \subset \C$ for which $S \cap \sigma(A, \Bc) = S \cap \sigma_{\rm point}(A, \Bc)$, we write $\mathfrak{m}_S(A, \Bc)\in \{ 0, 1, 2, \dots\} \cup \{ \infty\}$ for the sum of the algebraic multiplicities at each $\lambda \in S \cap \sigma_{\rm point}(A, \Bc)$.

For $\delta > 0$, we write
\begin{align*}
		\sigma_\delta(A, \Bc) = \sigma(A, \Bc) \cup \{ z \in \C :& (z - A) \text{ is boundedly invertible} \\ 
		& \text{ and } | (z - A)^{-1}| \geq \delta^{-1} \}
\end{align*}
for the \emph{$\delta$ pseudo-spectrum} of $A$, and 
\[
\rho_\delta(A, \Bc) = \C \setminus \sigma_\delta(A, \Bc)
\]
 for the corresponding $\delta$ \emph{pseudo-resolvent set}. Note that $\sigma_\delta(A, \Bc)$ is always closed and $\rho_\delta(A, \Bc)$ is always open. 
 
 For $r > 0$ and $p \in \C$, we write $B_r(p) \subset \C$ for the open ball of radius $r$ centered at $p$, and $\overline{B_r(p)}$ for its closure.

\subsection*{Setup for Section \ref{sec:householder}}

\medskip

In what follows, we will let $(\Bc_w, |\cdot|_w)$ be a Banach space, and $\Bc_s \subset \Bc_w$ a subspace equipped with a norm $|\cdot|_s$ with respect to which $(\Bc_s, |\cdot|_s)$ is also Banach. We will sometimes refer to these as  the ``weak'' and ``strong'' spaces, respectively. We will assume throughout that 
\begin{itemize}
	\item[(a)] $\Bc_s$ is dense in $(\Bc_w, |\cdot|_w)$; and 
	\item[(b)] we have the inequality 
	\[
	|\cdot|_w \leq |\cdot|_s \, .
	\]
\end{itemize}
Throughout, given an operator $A : \Bc_\tau \to \Bc_\tau, \tau \in \{ w, s\}$, we write $r_A^\tau(z) = |(z - A)^{-1}|_\tau$, where here $|\cdot|_\tau$ is the corresponding operator norm.

\subsection{Householder-type argument for operators}\label{subsec:householder}

The most general form of our results, Proposition \ref{prop:householder}, will be able to enclose spectrum of operators in the following class. 

\begin{definition}
	Let ${\bf L} = {\bf L}(\Bc_w, \Bc_s)$ denote the set of bounded linear operators $A : \Bc_w \to \Bc_w$ such that 
	\begin{itemize}
		\item[(i)] $A(\Bc_s) \subset \Bc_s$; and 
		\item[(ii)] the restriction $A|_{\Bc_s} : \Bc_s \to \Bc_s$ is a bounded linear operator w.r.t. $|\cdot|_s$. 
	\end{itemize}
	For $A \in {\bf L}$, we will write 
	\begin{align}\label{eq:mixNorm2}
		\| A\|_{s \to w} &= \sup_{v \in \Bc_s \setminus \{ 0 \} } \frac{| A (v)|_w}{|v|_s} 
	\end{align}
	for the \emph{mixed norm} of $A$. 
\end{definition}

Let $A, \tilde{A} \in {\bf L}$. The following argument 
allows to enclose the point spectrum of $A \in {\bf L}$ in $\Bc_s$ inside $\sigma_\delta(\tilde{A}, \Bc_w)$, under the assumption that $\| A - \tilde{A}\|_{s \to w}$ is sufficiently small. Below, for $\lambda \in \sigma_{\rm point}(A, \Bc_s)$ we write 
\[
\gamma_A(\lambda) := \inf\left\{ \frac{|f|_s}{|f|_w} : f \in \Bc_s \setminus \{ 0 \}, A f = \lambda f\right\} \, , 
\]
noting that in our setup, $\gamma_A(\lambda) \in [1,\infty)$. 

\begin{proposition}\label{prop:householder}
Let $A, \tilde{A} \in {\bf L}$ and let $\lambda \in \sigma_{\rm point}(A, \Bc_s)$. Then, for all 
	\[
	\delta \geq \gamma_A(\lambda) \| A - \tilde{A}\|_{s \to w}
	\]
	we have that
	\[
	\lambda \in \sigma_\delta(\tilde{A}, \Bc_w) \,. 
	\]
\end{proposition}

\begin{proof}
	The proof is a variation of the standard Householder argument \cite[Theorem I.2.2]{trefethenEmbree}, adapted to our setting in the presence of two different norms. 
	
	To start, if $\lambda \in \sigma(\tilde{A}, \Bc_w)$ then there is nothing to show, since in this case $\lambda \in \sigma_\delta(\tilde{A}, \Bc_w)$ for all $\delta > 0$. Otherwise, 
	 fix $f \in \Bc_s \setminus \{ 0\}$ such that 
	$A f = \lambda f$, and observe that
	\[
	A f - \tilde{A} f = \lambda f - \tilde{A} f \, , 
	\]
	hence
	\[
	(\lambda - \tilde{A})^{-1} (A - \tilde{A}) f = f 
	\]
	on viewing $(A - \tilde{A}) f$ as an element of $\Bc_w$. 
Taking $|\cdot|_w$-norms of both sides, 
\begin{align*}
	|f|_w & = | (\lambda - \tilde{A})^{-1} (A - \tilde{A}) f |_w \\
	& \leq | (\lambda - \tilde{A})^{-1}|_w \cdot \| A - \tilde{A}\|_{s \to w} \cdot |f|_s \, . 
\end{align*}
We conclude
\begin{align}\label{eq:intermediateStepResolventIneqHouseholder}
	|(\lambda - \tilde{A})^{-1}|_w \geq \left( \| A - \tilde{A}\|_{s \to w} \frac{|f|_s}{|f|_w} \right)^{-1} \, . 
\end{align}
Since $f$ in the $\lambda$-eigenspace of $A$ was arbitrary, we can replace this with the estimate
\begin{align*}
	|(\lambda - \tilde{A})^{-1}|_w \geq \left( \gamma_A(\lambda) \| A - \tilde{A}\|_{s \to w}  \right)^{-1} \, ,
\end{align*}
which completes the proof. 
\end{proof}

\begin{remark}\label{rmk:weakAssumpHouseholder2}
	As we will see below when controlling the quantity $\gamma_A(\lambda)$, it is often useful to think of $A$ as a linear operator on both $\Bc_w$ and $\Bc_s$. However, as seen in the foregoing proof, Proposition \ref{prop:householder} requires only that $\tilde{A} \in {\bf L}$ and $A : \Bc_s \to \Bc_s$ is bounded.
	
\end{remark}

\subsubsection*{Controlling $\gamma_A(\lambda)$}

The foregoing approach is quite general, making no special constraints on the operator $A$ aside from its approximation in the mixed norm $\| \cdot \|_{s \to w}$. The price paid, of course, is some control on the coefficient $\gamma_A(\lambda)$. 
To close Section \ref{subsec:householder} we remark on a special case, of interest for the transfer operator setting we adopt in later sections, for which an a priori estimate of  $\gamma_A(\lambda)$ can be given. 

\begin{assumption}\label{ass:DFLY}
	There exist real constants $C_1, C_2 > 0$ and $0 < \beta < M$ with the property that 
	\begin{align}\label{eq:DFLY}
		|A^n f|_s \leq C_1 \beta^n |f|_s + C_2 M^n |f|_w
	\end{align}	
	for all $f \in \Bc_s$. 
\end{assumption}

Equation \ref{eq:DFLY} is sometimes referred to as a \emph{Doeblin-Fortet-Lasota-Yorke} (DFLY) inequality, and is a powerful tool in the spectral theory of transfer operators \cite{doeblin1937chaines, marinescu1950theorie, lasota2004existence}. The following is essentially due to Hennion \cite{hennion1993theoreme}, building off of Nussbaum's characterization of essential spectrum \cite{nussbaum1970radius}. 
	
\begin{proposition}\label{prop:DFLY}
	Let $A \in {\bf L}$ and suppose Assumption \ref{ass:DFLY} holds for some $C_1, C_2 > 0$ and $0 < \beta < M$. Then the following hold. 
	\begin{itemize}
		\item[(a)] $\sigma_{\rm point}(A, \Bc_s) \subset \overline{B_M(0)}$. 
		\item[(b)] Assume moreover that $\Bc_s$ is compactly embedded in $\Bc_w$. Then, $\sigma(A, \Bc_s) \setminus \overline{B_\beta(0)}$ consists only of isolated point spectrum, and for all $\mu \in (\beta, M)$ it holds that $\sigma(A, \Bc_s) \setminus \overline{B_\mu(0)}$ is finite. 
	\end{itemize}
\end{proposition}

The following ties Assumption \eqref{ass:DFLY} to control over $\gamma_A(\lambda)$. 

\begin{lemma}\label{lem:controlgammalambda} \
\begin{itemize}
	\item[(a)]	Suppose that $A \in {\bf L}$ satisfies Assumption \ref{ass:DFLY} for some $C_1, C_2 > 0$ and $0 < \beta < M$.  Then, for $\lambda \in \sigma_{\rm point}(A, \Bc_s)$ with $|\lambda| > \beta$, it holds that 
		\begin{align} \label{eq:estGammaA2}\gamma_A(\lambda) \leq \inf_{n \geq 1}  \frac{C_2 M^n}{|\lambda|^n - C_1 \beta^n} \,. \end{align}
	\item[(b)] In the setting of (a), assume further that $C_1 = 0$. Then, 	\begin{align} \gamma_A(\lambda) \leq \frac{C_2 M}{|\lambda|} \,. \end{align} 
\end{itemize}	
\end{lemma}
\begin{proof}
	For (a), let $\lambda \in \sigma_{\rm point}(A, \Bc_s)$ and $|\lambda| > \beta$ and suppose $A f = \lambda f$ for some $f \in \Bc_s \setminus \{0 \}$. Then $A^n f = \lambda^n f$ for all $n \geq 1$, and by \eqref{eq:DFLY} we deduce 
	\[|\lambda|^n |f|_s = |A^n f|_s \leq C_1 \beta^n |f|_s + C_2 M^n |f|_w \,. \] 
	Collecting alike terms and taking an infimum over eigenfunctions $f$ and $n \geq 1$ completes the proof. Part (b) is immediate, although we note that when $C_1 = 0$ the infimum on the RHS of \eqref{eq:estGammaA2} is achieved at $n = 1$. 
\end{proof}

\begin{remark}
	Lemma \ref{lem:controlgammalambda}(b) applies when (i) $A(\Bc_w) \subset \Bc_s$ and (ii) $A : \Bc_w \to \Bc_s$ is bounded.  If these conditions are met, we can take $C_1 = 0, C_2 = 1$ and $M = \| A \|_{w \to s}$, where $\| \cdot \|_{w \to s}$ is defined analogously to \eqref{eq:mixNorm2}.  
\end{remark}

\subsection{Practical application to control of spectrum}\label{subsec:householderApp2}

We now change perspective somewhat, taking the view that $A \in {\bf L}$ is a fixed operator and $\tilde{A}$ is an approximation for which some spectral information is known. 
We present below a `practical' result for using Proposition \ref{prop:householder} to deduce properties of the spectrum of $A$ from those of the `known' operator $\tilde{A}$. 

Below, setup is as in the beginning of Section \ref{sec:householder}. We will assume moreover that $\Bc_s \subset \Bc_w$ is compactly embedded in all that follows. Recall that $\mathfrak{m}_K$ denotes the sum of algebraic multiplicities of eigenvalues in a region $K \subset \C$. Lastly, given a simple closed curve $\gamma \subset \C$, we will write $\Dc_\gamma$ for the (open) interior of $\gamma$ as in the Jordan curve theorem.
\begin{theorem}\label{thm:houseHolderSummary2}
	Let $A, \tilde A \in {\bf L}$  satisfy Assumption \ref{ass:DFLY} for the same values $C_1, C_2 > 0, 0 < \beta < M$. With $\mu \in (\beta, M)$ fixed, let 
	\[r = \inf_{n \geq 1} \frac{C_2 M^n}{\mu^n - C_1 \beta^n} \,. \]
	
	Let $\delta \geq r \| A - \tilde{A}\|_{s \to w}$, and 
	let $\gamma_1, \dots, \gamma_k$ be a collection of simple closed loops with the following properties. Below, we write $\gamma_0$ for the circle of radius $\mu$ centered at $0$, so that $\Dc_{\gamma_0} = B_\mu(0)$. 
	\begin{itemize}
		\item[(1)] The sets $\overline{\Dc_{\gamma_i}}, 0 \leq i \leq k$ are disjoint from each other; \edited{and}
		\item[(2)] $\sigma_\delta(\tilde{A}, \Bc_w) \subset \bigcup_{j = 0}^k \Dc_{\gamma_j}$. 
	\end{itemize}
	Then, 
	\begin{align}\label{eq:specContain2}
		\sigma(A, \Bc_s) \subset \bigcup_{j =0}^k \Dc_{\gamma_j} \, , 
	\end{align}
	and moreover\footnote{Note that by Proposition \ref{prop:DFLY}, $\sigma(A, \Bc_s) \setminus B_\mu(0), \sigma(\tilde{A}, \Bc_s) \setminus B_\mu(0)$ are both finite and consist of pure isolated point spectrum. In particular, the multiplicities on either side of \eqref{eq:multAgree2} are finite. }, 
	\begin{align}\label{eq:multAgree2}
		\mathfrak{m}_{\Dc_{\gamma_j}}(A, \Bc_s) = \mathfrak{m}_{\Dc_{\gamma_j}}(\tilde{A}, \Bc_s)
	\end{align}
	for all $1 \leq j \leq k$.

\end{theorem}
The proof of Theorem \ref{thm:houseHolderSummary2} and other proofs are  deferred to the end \edited{of} Section \ref{subsec:householderApp2}. 

In summary: Theorem \ref{thm:houseHolderSummary2} provides (i) an enclosure of the isolated point spectrum in $\sigma(A, \Bc_s) \setminus \overline{B_\mu(0)}$ within the interiors $\Dc_{\gamma_i}$ of the curves $\gamma_i,  1 \leq i \leq k$; and (ii) a guarantee that multiplicities within each $\Dc_{\gamma_i}$ will not change when passing from the `approximate' operator $\tilde{A}$ to the `exact' operator $A$.

Some comments on the hypotheses of Theorem \ref{thm:houseHolderSummary2} are in order: 

	\begin{itemize}
		\item[(i)] Taking \edited{the} view that $\tilde{A}$ is some exactly-known approximation to $A$, condition (1) is checkable in practice. Indeed, in applications to computer-assisted proof, one anticipates that the approximation $\tilde{A}$ is a finite rank operator, hence $\sigma(\tilde{A}, \Bc_w) = \sigma(\tilde{A}, \Bc_s)$, and it is enough to take the $\gamma_i, 1 \leq i \leq k$ to be small circles enclosing each point of $\sigma(\tilde{A}) \setminus \overline{B_\mu(0)}$. 
		\item[(ii)] Condition (2) is likewise checkable but computationally far more  difficult, a priori requiring establishing a rigorous ``heat map'' of $r_{\tilde{A}}^w$ across $\overline{B_M(0)}$. As it will turn out, validated enclosures of $r_{\tilde{A}}^w$ are the most expensive computations this method requires. The following auxiliary lemma addresses this concern by leveraging subharmonicity of the resolvent to significantly reduce the number of resolvent-norm computations. 
	\end{itemize}

\begin{lemma}\label{lem:subharmonic2}
	Let $B$ be bounded linear operator of a Banach space $(\Bc, |\cdot|)$, $\delta > 0$, and let $\zeta_1, \dots, \zeta_\ell$ be a finite collection of disjoint simple closed curves for which the corresponding sets $\overline{\Dc_{\zeta_i}}$ are disjoint. Assume moreover that 
	\begin{itemize}
		\item[(1)] $\sup_{z \in \zeta_i} r_B(z) \leq \delta^{-1}$ for all $1 \leq i \leq \ell$; and 
		\item[(2)] $\sigma(B, \Bc) \subset \bigcup_{i = 1}^\ell \Dc_{\zeta_i}$. 
	\end{itemize}
	Then, 
	\[\sigma_\delta(B, \Bc) \subset \bigcup_{i = 1}^\ell \Dc_{\zeta_i} \,. \]

\end{lemma}

Thus, in the setting of Theorem \ref{thm:houseHolderSummary2}, containment of the pseudospectrum as in condition (2) follows on checking the hypotheses of Lemma \ref{lem:subharmonic2}  with $B = \tilde{A}, \Bc = \Bc_w$, which amounts to estimation of $r_{\tilde{A}}^w(z)$ along a finite union of enclosing curves -- a substantial  reduction of computational effort in comparison with estimating $r_{\tilde{A}}^w(z)$ for $z$ ranging over a subset of $\C$ with nonempty interior.

\begin{remark}
		Especially for transfer operators, it will often be the case that the curves $\gamma_i$ isolating point spectrum of $A : \Bc_s \to \Bc_s$ will necessarily pass through essential spectrum of $A : \Bc_w \to \Bc_w$. Thus, it is not clear a priori that the weak resolvent bounds required for Lemma \ref{lem:subharmonic2} are ever feasible, even when the `approximate' operator $\tilde{A}$ is quite close to $A$ in the mixed norm $\| \cdot \|_{s \to w}$. 
		This critical feasibility issue is discussed further in Section \ref{subsec:diffSpaces}. 
\end{remark}

\subsubsection*{Proofs}

\begin{proof}[Proof of Theorem \ref{thm:houseHolderSummary2}]
	To start, the hypotheses and Lemma \ref{lem:controlgammalambda} imply that for $\lambda \in \sigma_{\rm point}(A, \Bc_s)$ with $|\lambda| \geq \mu$, it holds that $\gamma_A(\lambda) \leq r$, hence $\delta \geq r \| A - \tilde{A}\|_{s \to w} \geq \gamma_A(\lambda)\| A - \tilde{A}\|_{s \to w}$ for such $\lambda$.

	Turning now to \eqref{eq:specContain2}, let $\lambda \in \sigma(A, \Bc_s)$. There is nothing to show if $|\lambda| < \mu$. If $|\lambda| \geq \mu$, then $\lambda$ is an isolated eigenvalue of $A$ (Proposition \ref{prop:DFLY}) since $\mu > \beta$, and so $\lambda \in \sigma_\delta(\tilde{A}, \Bc_w)$ follows from Proposition \ref{prop:householder}. Hypothesis (2) now implies $\lambda \in \Dc_{\gamma_j}$ for some $1 \leq j \leq k$, and \eqref{eq:specContain2} follows. Note, in particular, that $\lambda \notin \gamma_j$ for any $0 \leq j \leq k$. 

	It remains to check \eqref{eq:multAgree2}, which is the bulk of the work we do here. It will be useful to record the following, a consequence of the spectral stability estimates for transfer operators in \cite{keller1999stability}; see also \cite[Appendix C]{demers2021transfer}. 
	\begin{definition}
		Let $C_1, C_2 > 0, 0 < \beta < M, \hat M > 0$. We write ${\bf L}(C_1, C_2, \beta, M, \hat M)$ for the set of operators $T \in {\bf L}$ for which (i) Assumption \ref{ass:DFLY} holds with the given constants $C_1, C_2, \beta, M$ 
		and, additionally, (ii) $|T|_w \leq \hat M$.  
	\end{definition}

	\begin{lemma}[{\cite[Corollary 1]{keller1999stability}}]\label{lem:DFLYSpecCty2}
		Let $T \in {\bf L}(C_1, C_2, \beta, M, \hat M)$ and $\mu \in (\beta, M)$. 
		Then, there exists $\delta_0 = \delta_0(T)$ such that for any $\delta \in (0,\delta_0(T))$ there exists $\epsilon_0 = \epsilon_0(T, \delta) > 0$ such that the following holds: for any $\lambda \in \sigma(T, \Bc_s) \setminus B_\mu(0)$ and for any $T' \in {\bf L}(C_1, C_2, \beta, M, \hat M)$ with $\| T - T'\|_{s \to w} \leq \epsilon_0$, one has that
		\begin{itemize}
			\item[(a)] $B_{\delta}(\lambda) \cap \sigma(T, \Bc_s) = \{ \lambda\}$; 
			\item[(b)] $\partial B_{\delta}(\lambda) \cap \sigma(T', \Bc_s) = \emptyset$; and 
			\item[(c)] $\mathfrak{m}_{B_{\delta}(\lambda)}(T, \Bc_s) = \mathfrak{m}_{B_\delta(\lambda)}(T', \Bc_s)$.
		\end{itemize}
	\end{lemma}

	To complete the proof of \eqref{eq:multAgree2}, let
	\[A_t = t A + (1 - t) \tilde{A} \, , \quad t \in [0,1] \, .\]
	Note that by hypothesis, Assumption \ref{ass:DFLY} holds for the family $(A_t)_{t \in [0,1]}$ with the same constants $C_1, C_2, \beta, M$. Moreover, $|A_t|_w \leq \hat M$ where $\hat M = \max\{ |A|_w, |\tilde{A}|_w\}$. 
	Thus, $t \mapsto A_t$ is a path into ${\bf L}(C_1, C_2, \beta, M, \hat M)$, which is continuous in the mixed norm $\| \cdot \|_{s \to w}$ in view of the identity $\| A_t - A_{t'}\|_{s \to w} = (t-t') \| A - \tilde{A}\|_{s \to w}$ for $t, t' \in [0,1]$. In particular, at $t' = 0$ it follows from Proposition \ref{prop:householder} with $A$ replaced by $A_t$ that
	\begin{align}\label{eq:homotopySpecContain2}
		\sigma(A_t, \Bc_s)\setminus B_\mu(0) \subset \sigma_\delta(\tilde{A}, \Bc_w) \, . 
	\end{align}

	Continuing, define 
	\[\rho_0 = \min_{0 \leq j \leq k} \min_{z \in \gamma_j} \dist(z, \sigma_\delta(\tilde{A}, \Bc_w)) \, . \]
	For $t \in [0,1]$, let $\delta_1(A_t) = \min\{ \delta_0(A_t), \rho_0\}$. 
	It now follows from compactness of $[0,1]$ that there exist
	\[0 =: t_0 < t_1 < \dots < t_{N-1} < t_N := 1\]
	with the property that for $1 \leq i \leq N-1$ and for $t \in [t_{i-1}, t_{i + 1}]$, one has that $\| A_t - A_{t_i}\|_{s \to w} < \epsilon_0(A_{t_i}, \delta_1(A_{t_i}))$.
	From this, we conclude that for any $1 \leq i \leq N-1$ 
	and $\lambda \in \sigma(A_{t_i}, \Bc_s) \setminus B_\mu(0)$, it holds that for $t \in [t_{i-1}, t_{i + 1}]$, \begin{itemize}
		\item[(i)] $\sigma(A_t, \Bc_s) \cap \partial B_{\delta_1(A_{t_i})}(\lambda) = \emptyset$; and 
		\item[(ii)] $\mathfrak{m}_{B_{\delta_1(A_{t_i})}(\lambda)}(A_t, \Bc_s) = \mathfrak{m}_{B_{\delta_1(A_{t_i})}(\lambda)}(A_{t_i}, \Bc_s)$
	\end{itemize}
	By \eqref{eq:homotopySpecContain2}, such $\lambda$ are contained in $\sigma_\delta(\tilde{A}, \Bc_w)$, and so by our construction it follows that each such ball $B_{\delta_1(A_{t_i})}(\lambda)$ is contained in $\mathcal{D}_{\gamma_j}$ for some $1 \leq j \leq k$. 
	It now follows from point (ii) above that $\mathfrak{m}_{\Dc_{\gamma_j}}(A_t, \Bc_s)$ is constant over $t \in [t_{i-1}, t_{i + 1}]$. This completes the proof of \eqref{eq:multAgree2}. 
\end{proof}

\begin{proof}[Proof of Lemma \ref{lem:subharmonic2}]
	Let 
	\[\Omega_0 = \C \setminus \bigcup_{i = 1}^\ell \overline{\Dc_{\zeta_i}} \, , \]
	noting that $\Omega_0$ is an unbounded, connected open set with boundary $\cup_i \zeta_i$. 
	We seek to show that $\sigma_\delta(B, \Bc) \cap \Omega_0 =\emptyset$, i.e., $z \in \Omega_0$ implies $r_B(z) < \delta^{-1}$.

	It is a standard fact that $z \mapsto r_B(z)$ is subharmonic along the resolvent set $\rho(T, \Bc)$; hypothesis (2) implies $\rho(B, \Bc) \supset \Omega_0$, hence $z \mapsto r_B(z)$ is subharmonic along $\Omega_0$. By the maximum principle for subharmonic functions, 
	\[r_B(z) < \sup_{z \in \cup_i \zeta_i} r_B(z) \quad \text{ for all } \quad z \in \Omega_0 \, ,  \]
	provided $z \mapsto r_B(z)$ is nonconstant along $\Omega_0$. The proof is complete on applying the estimate in hypothesis (1). 	

	It remains to check $z \mapsto r_B(z)$ is nonconstant along $\Omega_0$. This follows from \cite[Proposition 1]{globevnik1976norm}, the argument for which we reproduce here. To start, we observe that the von Neumann series 
	\[(z - B)^{-1} = z^{-1} \sum_{n = 0}^\infty z^{-n} B^n\]
	converges in operator norm for all $|z| > |B|$, hence $r_B(z) \to 0$ as $z \to \infty$. We conclude that if $r_B(z)$ were constant along $\Omega_0$, one would have that $r_B(z) \equiv 0$, which is incompatible with the invertibility of $(z - B)^{-1}$ for $z \in \rho(B, \Bc)$. This completes the proof. 
\end{proof}

\begin{remark}
	We note a subtlety here, that it is not automatic that $r_B$ is nonconstant along arbitrary open subsets-- see \cite[Theorem 3.1]{shargorodsky2008level} for a counterexample and further discussion. In the proof of Lemma \ref{lem:subharmonic2} above, we have used crucially that the region $\Omega_0$ is unbounded. 
\end{remark}

\subsection{Avoiding uncontrolled weak resolvent blowup} \label{subsec:diffSpaces}

To frame the discussion, let us take on the following natural abstract setup for the use of Theorem \ref{thm:houseHolderSummary2} to isolate point spectrum of an operator $A$. Notation and the pair of spaces $\Bc_s, \Bc_w$ are as in the beginning of Section \ref{sec:householder}.
	
	\begin{assumption}\label{ass:setup2-3}
		The operator $A$ belongs to ${\bf L}$, and $(\Pi_n)$ is a sequence of finite-rank projection operators $\Bc_w \to \Bc_w$ with the following properties. 
		\begin{itemize}
			\item[(i)] $\Pi_n(\Bc_s) \subset \Bc_s$;  
			\item[(ii)] the sequence of truncated operators \[\tilde{A}_n := \Pi_n A \Pi_n \] 
			is such that $\Delta_A(n) := \| \tilde{A}_n - A\|_{s \to w} \to 0$ as $n \to \infty$; and 
			\item[(iii)] $A$ and $(\tilde{A}_n)$ all satisfy Assumption \ref{ass:DFLY} for some fixed constants $C_1, C_2 > 0, 0 < \beta < M$. 
			\item[(iv)] It holds\footnote{The constraint $\| \Pi_n\|_w = 1$ is not critically important in what follows, but rather is made for the sake of brevity.} that $\| \Pi_n\|_w = 1$ and $\| A \|_w \leq 1$. In particular, $\| \tilde{A}_n\|_w \leq 1$ for all $n$. 
		\end{itemize}
	\end{assumption}

Note that the uniform DFLY inequality condition is natural in this context: it is the underlying hypothesis
to the Keller-Liverani spectral stability theorem \cite{keller1999stability}. Violating 
this condition provides examples of systems where the spectrum is not stable under perturbations
\cite{keller1982stochastic}.
	
	Let us suppose that, given some collection $\{ \lambda_i\}_{i = 1}^K \subset \sigma(A; \Bc_s)$ of isolated point spectrum for $A$, one can find curves $\gamma_i, 0 \leq i \leq n$ as in Theorem \ref{thm:houseHolderSummary2} such that $\gamma_i$ encircles $\lambda_i$ for $1 \leq i \leq n$, while $\gamma_0$ encircles a possible essential spectral block around the origin. 
	Using subharmonicity of the resolvent norm as in Lemma \ref{lem:subharmonic2}, Theorem \ref{thm:houseHolderSummary2} applies for the collection $\{ \lambda_i\}$ if, for some $n$, we can derive a resolvent upper bound 
	\begin{align}\label{eq:estGamma2-3}\max_{0 \leq i \leq n} \sup_{z \in \gamma_i} r_{\tilde{A}_n}^w(z) \leq \frac{1}{ r \Delta_A(n) } \, , \end{align}
	with $r$ as in Theorem \ref{thm:houseHolderSummary2}. 

	Since $\Delta_A(n) \to 0$ as $n \to \infty$, one might hope that realizing \eqref{eq:estGamma2-3} is `feasible' on taking $n$ sufficiently large. 
	However, a priori there is no guarantee for this due to the fact that, particularly for transfer operators, 
	\emph{point spectrum of interest in $\Bc_s$ is surrounded by essential spectrum for $\Bc_w$.} Thus, it is often unavoidable that the curves $\gamma_i$ encircling point spectrum will necessarily meet essential spectrum\footnote{See, e.g., \cite{gundlach2003sharp} for sharp essential spectral radii for transfer operators of uniformly expanding maps acting on various function spaces.} of $A : \Bc_w \to \Bc_w$, leading to possible weak resolvent blowup, i.e., 
	\[r^w_{\tilde{A}_n}(z) \to \infty \quad \text{ for } z \in \gamma_i\quad  \text{ as } n \to \infty \,. \]
	If this blowup is faster than the growth in $\big(\Delta_A(n)\big)^{-1}$ as $n \to \infty$, there is no hope for realizing \eqref{eq:estGamma2-3}. 

	Our goal in this section will be to exhibit some natural conditions under which the blowup scenario described above is \emph{not too fast}, ensuring  \eqref{eq:estGamma2-3} is indeed feasible on taking $n$ sufficiently large. Our result to this effect is as follows. Below, we set 
	\begin{gather*}
		\Bc_n = \operatorname{Range}(\Pi_n) \subset \Bc_s \, ,  \\
		\mathcal{C}(n) = \max\left\{\frac{|v|_s}{|v|_w} : v \in \Bc_n \setminus \{ 0 \} \right\} \, , 
	\end{gather*}
	noting $\mathcal{C}(n) < \infty$ for all $n$ since $\Pi_n$ is finite-rank.
	To simplify estimates, we assume below that $M = 1$, which can always be achieved on re-scaling the operator $A$. 
	\begin{lemma}\label{lem:aPrioriWeakResolvent2}
		With notation as in Assumption \ref{ass:DFLY} and assuming further that $M = 1$, let $\mu > \beta$ and let \[z \in \C \setminus \left(B_\mu(0) \cup \sigma(A; \Bc_s) \cup \sigma(\tilde{A}_n)\right) \,. \]
		Let $n \geq 1$. Then, 
		\begin{align*}
			r^w_{\tilde{A}_n}(z) \leq {\mathcal{C}(n)^{\mathfrak{q}} \over 1 - \hat C  r^s_A(z) \Delta_A(n) \mathcal{C}(n)^{\mathfrak{q}}} \left(\frac{1}{1 - \mu} + \hat C r^s_A(z)  \right) \,.
		\end{align*}
		where $\mathfrak{q} := \frac{|\log \mu|}{|\log \beta|} < 1$ and $\hat C = C_1 + C_2$, and assuming the RHS denominator is $> 0$. 
	\end{lemma}

Since $\mathcal{C}(n) \to \infty$ as $n \to \infty$ and since the exponent $\mathfrak{q}$ is less than 1, Lemma \ref{lem:aPrioriWeakResolvent2} implies that our desired weak resolvent estimate \eqref{eq:estGamma2-3} holds for all $n$ sufficiently large if, \edited{e.g.,} 
\[\Delta_A(n) \cdot \mathcal{C}(n) = O(1) \qquad \text{ as } \quad n \to \infty \,, \]
\edited{i.e., $\Delta_A(n) \cdot \mathcal{C}(n)$ is bounded as $n \to \infty$. Assuming this growth condition,} our desired method is \edited{(at least theoretically)} feasible, \edited{in the sense that} for any $\epsilon > 0, \edited{\mu > \beta,}$ there exists an $n$ such that
\[
\sigma(A) \cap B_{\mu}(0)^c \subseteq \left\{z : r^w(z, \tilde{A}_n)>\frac{1}{\epsilon}\right\} \cap B_{\mu}(0)^c \, .
\] 

\begin{example}
In the case of the Ulam discretization $\Pi_n$ as in Example \ref{exam:BVapprox}, Assumption \ref{ass:setup2-3} holds if $A = \Lc_T$ is the transfer operator of a piecewise-linear uniformly expanding map on $[0,1]$, with constants $C_1=1$, $\beta<1$, $M=1$ and $C_2$.
\edited{We now set about bounding $\Delta_A(n) \cdot \mathcal{C}(n)$. From} \cite[Lemma 3.5]{galatolo2023general}, \edited{we have}
\begin{align*}
\Delta_A(n) &= \|\Pi_n \Lc_T \Pi_n - \Lc_T\|_{BV\to L^1} \\ 
& \leq \|\Pi_n \Lc_T (\Pi_n - I)\|+\|(\Pi_n - I) \Lc_T\|\leq \frac{1}{2n}+\frac{\beta+C_2}{2n} \, , 
\end{align*}
\edited{while from} \cite[Lemma 6.3]{galatolo2023general} we have that $\mathcal{C}(n) = 2n$. \edited{In total, $\Delta_A(n) \cdot \mathcal{C}(n) = O(1)$ as $n \to \infty$, as desired.}
\end{example}

\begin{remark}
	Another natural example is that of Galerkin truncation, i.e., truncation to a finite number of Fourier modes. This is the subject of Section \ref{sec:analyticframework} below, where Galerkin truncation is applied to the study of resonances of analytic expanding maps. 
	We caution, however, that the developments in Section \ref{sec:analyticframework} are distinct from those given here -- it is convenient to work there with three spaces, not two. Strictly speaking, the discussion of feasibility given here does not apply to the functional framework of Section \ref{sec:analyticframework}. That being said, the fact that our rigorous computer-assisted proof carries through is, in the authors' view, a reasonable demonstration of `feasibility' of the method in that setting.
\end{remark}

	\subsubsection*{Proof of Lemma \ref{lem:aPrioriWeakResolvent2}}

	The argument given below is inspired by \cite{keller1999stability}. 
	We start with some preliminaries. The following is standard. 

	\begin{lemma}\label{lem:resolvIdentityN}
		For a linear operator $B$ and for $z \in \rho(B)$, the resolvent set of $B$, we have that 
		\[R(z,B) := (z - B)^{-1} = z^{-1} \sum_0^{N-1} z^{-\ell} B^\ell + z^{-N} B^N R(z, B) \, .\]
	\end{lemma}
	
	The next estimate is evident from Assumption \ref{ass:setup2-3}. Recall below that we assume $A, \tilde{A}_n$ all satisfy Assumption \ref{ass:DFLY} with $M = 1$. 
	
	\begin{lemma}\label{lem:LYregularize}
		Let $n \geq 1$. Then, for 
		\[N \geq \left\lceil {\log \mathcal{C}(n) \over |\log \beta|} \right\rceil \]
		it holds that 
		\[\| \tilde{A}_n^N \|_{w \to s} \leq \hat C := C_1 + C_2 \,. \]
	\end{lemma}

	\begin{proof}[Proof of Lemma \ref{lem:aPrioriWeakResolvent2}]

		We use Lemma \ref{lem:resolvIdentityN} to estimate  as follows: for $f \in \Bc_w$, $|z| \geq \mu$,  
	\[|R(z, \tilde{A}_n) f|_w \leq \underbrace{|z|^{-1} \sum_0^{N-1} |z|^{-\ell} |\tilde{A}_n^\ell f|_w}_{I} + \underbrace{|z|^{-N} | R(z, \tilde{A}_n) \tilde{A}_n^N f|_w}_{II} \, . \]
	Here, $N$ is as in Lemma \ref{lem:LYregularize}. 
	
	For I, we use that $|\tilde{A}_n|_w \leq 1$ to estimate 
	\begin{align*}
		I \leq \mu^{-1} \sum_0^{N-1} \mu^{-\ell} |f|_w = \mu^{-1} {\mu^{-N} - 1 \over \mu^{-1} - 1} |f|_w \leq  {\mu^{-N} \over 1 - \mu } |f|_w \,. 
	\end{align*}
	For II, 
	\begin{align*}
		|R(z, \tilde{A}_n) \tilde{A}^N_n f|_w & \leq \underbrace{|R(z, A) \tilde{A}^N_n f|_w}_{II.A} + \underbrace{|\big(R(z, \tilde{A}_n) - R(z, A) \big) \tilde{A}^N_n f|_w }_{II.B} \,. 
	\end{align*}
	For this RHS, we estimate 
	\begin{align*}
		II.A & \leq |R(z, A) \tilde{A}^N_n f|_s \leq r^s_A(z) \|\tilde{A}^N_n\|_{w \to s} |f|_w \\ 
		& \leq \hat C r^s_A(z) |f|_w \, , 
	\end{align*}
	and, with the second resolvent identity $R(z, \tilde{A}_n) - R(z, A) = R(z, \tilde{A}_n) (\tilde{A}_n - A) R(z, A)$, 
	\begin{align*}
		II.B & \leq r^w_{\tilde{A}_n}(z) \| \tilde{A}_n - A\|_{s \to w} r^s_A(z) \| \tilde{A}^N_n\|_{w \to s} |f|_w \\ 
		& \leq \hat C r^w_{\tilde{A}_n}(z) r^s_A(z) \edited{\Delta_A(n)} |f|_w  \,. 
	\end{align*}
	Finally, by our choice of $N$, we note that $N \leq 1 + \log \mathcal{C}(n) / |\log \beta|$, hence 
	\[\mu^{-N} \leq \mu^{-1} \mathcal{C}(n)^{|\log \mu|/|\log \beta|} \,. \]
	
	Putting all these estimates together, dividing through by $|f|_w$ and taking $\sup_{f \in \Bc_w \setminus \{ 0 \}}$, we conclude 
	\begin{align*}
		r^w_{\tilde{A}_n}(z) \leq \mathcal{C}(n)^{|\log \mu|/|\log \beta|} \left(\frac{1}{1 - \mu} + \hat C r^s_A(z) + \hat C r^w_{\tilde{A}_n}(z) r^s_A(z) \edited{\Delta_A(n)} \right) \,.
	\end{align*}
	Rearranging to isolate $r^w_{\tilde{A}_n}$ completes the proof. 
\end{proof}

%% file: analytic.tex
    We present in this section a functional-analytic framework for the application of ideas in Section \ref{sec:householder} to the transfer operators of a class of analytic expanding maps of the circle. 
    It is standard procedure in this setting to work with spaces of holomorphic functions see, e.g., \cite{mayer1991continued}. 
    
    The plan is as follows: after some preliminaries on function spaces (Section \ref{subsec:prelim3}) and analytic expanding maps (Section \ref{subsec:analyticExp3}), we present in Section \ref{subsec:householderExpanding3} the details of applying the Householder-type argument to this setting. 

\subsection{Preliminaries} \label{subsec:prelim3}

Throughout, we regard $S^1$ as the unit circle in the complex plane $\C$. Given a complex-valued function $f \in L^1(S^1)$ and wavenumber $k \in \Z$, we write 

\[c_k(f) = \frac{1}{2 \pi i}\int_{S^1} \frac{f(z)}{z^{k + 1}} dz \]
for the $k$-th Fourier coefficient. When it is clear from context, we will write $c_k = c_k(f)$. 

\begin{definition}
	The \emph{Wiener Algebra} $\Ac_0$ is defined to be the space of $f \in L^1(S^1)$ with absolutely convergent Fourier series, endowed with norm 
	\[\| f\|_{\Ac_0} = \sum_{k \in \Z} |c_k| \,. \]
	For $\eta > 0$, we define the related exponentially-weighted spaces $\Ac_\eta \subset \Ac_0$ as the set of $f \in \Ac_0$ for which the norm 
	\[\| f\|_{\Ac_\eta} = \sum_{k \in \Z} e^{2 \pi |k| \eta} |c_k|\]
	is finite. 
\end{definition}

We collect below some easily-derived facts about the scale of spaces $(\Ac_\eta)_{\eta \geq 0}$. 

\begin{lemma} \label{lem:holomSpaceProps}\ 
	\begin{itemize}
		\item[(a)] For all $\eta \geq 0$, $(\Ac_\eta, \| \cdot \|_\eta)$ is a Banach space. 
		\item[(b)] For $\eta > 0$ and $f \in \Ac_\eta$, the formal power series 
		\[ \sum_{k \in \Z} c_k z^k\]
		converges and gives a holomorphic function $\tilde{f}$ in the annulus $\mathfrak{A}_r = \{ r^{-1} < |z| < r\}$, where $r = e^{2 \pi \eta}$. Moreover, $\tilde{f}$ extends $f$, in that $\tilde{f} = f$ almost-surely along $S^1$, and \edited{$|\tilde{f}(z)| \leq \| f\|_{\Ac_\eta}$} uniformly along $z \in \mathfrak{U}_{e^{2 \pi \eta}}$. 
		\item[(c)] For all $0 \leq  \eta < \alpha$, it holds that $\Ac_\alpha \subset \Ac_\eta$ is compactly embedded and that $\| f \|_{\Ac_\eta} \leq \| f \|_{\Ac_\alpha}$ for all $f \in \Ac_{\alpha}$. 
	\end{itemize}
\end{lemma}

\begin{proof}[Proof sketch]
	Items (a) and (b) are standard results. Compact embedding as in (c) is a consequence of Montel's theorem \cite[Theorem VII.2.9]{conway1973functions}. 
\end{proof}

\begin{remark}
    We note that while functions in $\Ac_0$ are automatically continuous, this space does contain some exotic functions such as the Weierstrass $W$-function. In general, membership in $\Ac_0$ does not provide any control on the modulus of continuity, since the pure Fourier modes $z \mapsto z^n, n \in \Z$ are all unit vectors in $\Ac_0$, and $\Ac_0$ is not contained in any positive-regularity Sobolev space. 

\end{remark}

We conclude Section \ref{subsec:prelim3} with the following useful interpolation inequality. 

\begin{lemma}\label{lemma:interpolation}
	Let $0 \leq \eta < \alpha$ and let $f\in \mathcal{A}_{\alpha}$. Then, 
	\[
	\|f\|_{ 
		\Ac_\eta}\leq \|f\|_{\Ac_0}^{\frac{\alpha-\eta}{\alpha}} \|f\|_{ \Ac_\alpha}^{\frac{\eta}{\alpha}}.
	\]
	\end{lemma}
	\begin{proof}
	If $\eta = 0$, there is nothing to prove. If $\eta > 0$, then by H\"older's inequality we have
	\[
	\sum_{k} |c_k| e^{2\pi |k| \eta} = \sum_{k} |c_k|^{\frac{\alpha-\eta}{\alpha}} (|c_k| e^{2\pi |k| \alpha})^{\frac{\eta}{\alpha}}\leq \left(\sum_k |c_k|\right)^{\frac{\alpha-\eta}{\alpha}}\left(\sum_k |c_k| e^{2\pi |k| \alpha}\right)^{\frac{\eta}{\alpha}} \, , 
	\]
	completing the proof. 
	\end{proof}

\subsection{Analytic expanding maps of $S^1$}\label{subsec:analyticExp3}

The following is a treatment of transfer operators of analytic expanding maps on analytic function spaces. Much of the following is standard and included to make the manuscript self-contained. See, e.g., \cite{mayer1991continued} for further information and perspectives. 



\begin{definition}
    Let $T : S^1 \to S^1$ be a $C^1$ mapping. We say that $T$ is \emph{uniformly expanding} if $|T'| > 1$. Given $T$, its \emph{transfer operator} is the linear operator $\Lc_T : L^1(S^1) \to L^1(S^1)$ given by 
    \[\Lc_T \varphi (x) = \sum_{y \in T^{-1} x} \frac{\varphi(y)}{|T'(y)|} \,. \]
\end{definition}

We are interested in analytic expanding maps, and in what follows will view $\Lc_T$ as acting on the spaces $\Ac_\eta$ for appropriately chosen $\eta$. Key to our approach is the following expansion assumption: 

\begin{assumption}\label{ass:domainExpansion3}
    The mapping $T: S^1 \to S^1$ is analytic, uniformly expanding, and there exist $0 < \eta < \rho$ and $\epsilon > 0$ such that 
    \begin{itemize}
        \item[(i)] $T$ admits an analytic extension onto $\mathfrak{U}_{e^{2 \pi \eta} + \epsilon}$; and 
        \item[(ii)] $\mathfrak{U}_{e^{2 \pi \rho}} \cap T \left(\partial \mathfrak{U}_{e^{2 \pi \eta}}\right) = \emptyset$. 
    \end{itemize}
\end{assumption}
In particular, Assumption \ref{ass:domainExpansion3}(ii) ensures the image under $T$ of the annulus $\mathfrak{U}_{e^{2 \pi \eta}}$ contains the larger annulus $\mathfrak{U}_{e^{2 \pi \rho}}$. We note that Assumption \ref{ass:domainExpansion3} always holds for some choice of $\eta, \rho$, c.f. \cite[Lemma 2.2]{slipantschuk2013analytic}. 


\begin{lemma}\label{lem:regImrpovement3}
    Let Assumption \ref{ass:domainExpansion3} hold, and let $\alpha \in (\eta, \rho)$. Then, $\Lc_T(\Ac_\eta) \subset \Ac_\alpha$, and $\Lc_T : \Ac_\eta \to \Ac_\alpha$ is a bounded linear operator with 
    \begin{align}
\|\Lc_T\|_{\Ac_\eta \to \Ac_\alpha}\leq 1+\frac{2}{e^{2 \pi (\rho - \alpha)} - 1} \,. 
    \end{align}
\end{lemma}

\begin{proof}
    \edited{Without loss, assume $T$ winds clockwise around $S^1$; in the case of counterclockwise winding, the formulas below hold with a minus sign.}

    Fix $\varphi \in \Ac_\eta$. For $k \in \Z$ we compute 
    \begin{align}
        c_k(\Lc_T \varphi) & = \frac{1}{2 \pi i} \int_{S^1} \frac{\Lc_T \varphi}{z^{k + 1}} dz = \frac{1}{2 \pi i} \int_{S^1} \frac{\varphi(z)}{\left(T(z)\right)^{k + 1}} dz
    \end{align}
    by the usual change-of-variables formula. Specializing to the case $k \geq 0$, we can contour shift to obtain
    \begin{align}
        c_k(\Lc_T \varphi) = \frac{1}{2 \pi i} \int_{e^{2 \pi \eta} S^1} \frac{\varphi(z)}{\left(T(z)\right)^{k + 1}} dz \, , 
    \end{align}
    where above for $r > 0$ we write $r S^1$ for the circle of radius $r$ centered at the origin (having used that $T$ is analytic just beyond $\mathfrak{U}_{e^{2 \pi \eta}}$). We now estimate 
        \begin{align*}|c_k(\Lc_T \varphi)| & \leq e^{2 \pi \eta} \left(\sup_{z \in \mathfrak{U}_{e^{2 \pi \eta}}} |\varphi(z)| \right) \left(\inf_{z \in e^{2 \pi \eta} S^1} |T(z)|\right)^{-(k + 1)} \\ 
        & \leq e^{2 \pi \eta} \| \varphi\|_{\Ac_\eta} e^{- 2 \pi \rho(k + 1)} \leq  \| \varphi\|_{\Ac_\eta} e^{- 2 \pi \rho k} \,. 
        \end{align*}
    having bounded the first parenthetical term as in Lemma \ref{lem:holomSpaceProps}(b) and using Assumption \ref{ass:domainExpansion3}(ii) to control $|T(z)|$ from below.  
    
    An analogous argument holds for $k < 0$ on contour shifting to the smaller circle $e^{- 2 \pi \eta} S^1$, giving $|c_k(\Lc_T \varphi)| \leq  \| \varphi\|_{\Ac_\eta} e^{- 2 \pi \rho |k|}$ for all $k \in \Z$. Plugging this into the definition of $\| \Lc_T \varphi\|_{\Ac_\alpha}$ now implies the desired estimate. 
\end{proof}

The following are immediate. 
\begin{corollary}\label{cor:compactnessTfer3} \ 
    \begin{itemize}
        \item[(a)] $\Lc_T$ is bounded and compact when viewed as a bounded linear operator $\Ac_\alpha \to \Ac_\alpha$ or $\Ac_\eta \to \Ac_\eta$. 
        \item[(b)] $\sigma(\Lc_T, \Ac_\eta) = \sigma(\Lc_T, \Ac_\alpha)$, and both consist of pure point spectrum away from $0$. 
    \end{itemize}
\end{corollary}
When Assumption \ref{ass:domainExpansion3} holds and when otherwise clear from context, we will abuse notation and write $\sigma(\Lc_T)$ for the spectrum of $\Lc_T$ on either of $\Ac_\alpha$ or $\Ac_\eta$. 

\begin{remark}
    Conditions like Assumption \ref{ass:domainExpansion3} are known to imply refinements of compactness, e.g., being trace-class in $H^2(U)$ or nuclear in $H^p(U), p \in [1,\infty)$, where here $H^p(U)$ is the $p$-Hardy space on an appropriately-chosen open set $U \subset \mathbb{C}$; see, e.g., \cite{mayer1991continued} for further details. However, we will not make direct use of these refinements in this manuscript.
    
\end{remark}


We will require the following in our intended application of the Householder-type result (Proposition \ref{prop:householder}) to entraining this point spectrum. 

\begin{lemma}\label{lem:alphaToZeroEst3}
    Let Assumption \ref{ass:domainExpansion3} hold, and let $\alpha \in (\eta, \rho)$. Let $\lambda \in \sigma(\Lc_T) \setminus \{ 0 \}$ with corresponding eigenfunction $f \in \Ac_\alpha$. Then, we have the estimate 
\[
\|f\|_{\Ac_\alpha} \leq \left(\frac{1}{|\lambda|}\right)^{\frac{\alpha}{\alpha-\eta}} \left(1+\frac{2}{e^{2 \pi (\rho - \alpha)} - 1}\right)^{\frac{\alpha}{\alpha-\eta}} \|f\|_{\Ac_0} \,. 
\]
\end{lemma}

\begin{proof}
Fix $f \in \Ac_\eta$. Combining Lemmas \ref{lem:regImrpovement3} and \ref{lemma:interpolation}, 
we compute 
\begin{align*}
\|\Lc_T f\|_{\Ac_\alpha} & \leq \left(1+\frac{2}{e^{2 \pi (\rho - \alpha)} - 1}\right) \|f\|_{\Ac_\eta} \\ 
& \leq \left(1+\frac{2}{e^{2 \pi (\rho - \alpha)} - 1}\right) \|f\|_{\Ac_0}^{\frac{\alpha-\eta}{\alpha}}\|f\|_{\Ac_\alpha}^{\frac{\eta}{\alpha}}.
\end{align*}
Using $\Lc f = \lambda f$ and combining like terms gives  
\[|\lambda| \| f \|_{\Ac_\alpha}^{\frac{\alpha - \eta}{\alpha}} \leq   \left(1+\frac{2}{e^{2 \pi (\rho - \alpha)} - 1}\right) \| f \|_{\Ac_0}^{\frac{\alpha - \eta}{\alpha}} \, , \]
from which the desired conclusion follows. 
\end{proof}

Continuing, we present a discretization scheme for approximating transfer operators by those of finite rank. 
\begin{definition}
    Let $K \geq 1$. For $f \in L^1(S^1)$, the \emph{Galerkin projection} $\Pi_K f$ is given by 
    \begin{align}
        \Pi_K f (z)= \sum_{k = -K}^{K} c_k(f) z^k \, . 
    \end{align}
\end{definition}
We use repeatedly the following.
    \begin{lemma}\label{lem:galerkinEst3} For $\alpha > \eta > 0$, we have
        \[
        \| I - \Pi_K \|_{\Ac_\alpha \to \Ac_\eta} \leq e^{- 2 \pi K (\alpha - \eta)} \,. 
        \]
        \end{lemma}

        \begin{proof}
        By direct computation, for $f \in \Ac_\alpha$ we have 
        \begin{align*}
            \| f - \Pi_K f \|_{\Ac_\eta} &  \leq \sum_{|k| > K} |c_k| \edited{e^{2 \pi \eta |k|}} \\ 
            & \leq e^{- 2 \pi (\alpha - \eta) K} \sum_{|k| > K} |c_k| e^{2 \pi \alpha |k|} \\ 
            & \leq e^{- 2 \pi (\alpha - \eta) K} \| f \|_{\Ac_\alpha} \,. \qedhere
        \end{align*}
        \end{proof}
        Galerkin \edited{truncations} naturally give rise to the finite-rank approximants
        \begin{align}
            \Lc_{T, K} = \Pi_K \circ \Lc_T \circ \Pi_K \, , 
        \end{align}
        which we can view either as linear operators on $\Ac_\eta$-type spaces, or as linear operators on $\C^{2K +1}$ in the Fourier basis $\{ z^n\}_{n \in \Z}$. Note that 
        \begin{align}\label{eq:boundGalerkin}
            \| \Pi_K \|_{\Ac_\eta} = 1 \qquad \text{ for all } \eta \geq 0 \,. 
        \end{align}

        \begin{lemma}\label{lem:Galerkinapprox3}
            Let Assumption \ref{ass:domainExpansion3} hold, and let $\alpha \in (\eta, \rho)$. For $K \geq 1$, we have that
        \[
        \|\Lc_T -\Lc_{T,K} \|_{\Ac_\alpha \to \Ac_0}\leq \left(1+\frac{2}{e^{2 \pi (\rho - \alpha)} - 1}\right)\left(e^{-2\pi K\alpha}+e^{-2\pi K(\alpha-\eta)}\right) \, . 
        \]
        \end{lemma}
        \begin{proof}
        Throughout, we write $\Lc = \Lc_T, \Lc_K = \Lc_{T, K}$. We estimate
        \begin{align} \label{eq:discEst3} \| \Lc - \Pi_K \Lc \Pi_K\|_{\Ac_\alpha \to \Ac_0} \leq \| (I- \Pi_K) \Lc \|_{\Ac_\alpha \to \Ac_0} + \| \Pi_K \Lc (I - \Pi_K) \|_{\Ac_\alpha \to \Ac_0} \, . \end{align}

        For the first RHS term in \eqref{eq:discEst3}, combining Lemmas \ref{lem:regImrpovement3} and \ref{lem:galerkinEst3} gives
            \begin{align*}
                \| (I - \Pi_K ) \Lc \|_{\Ac_\alpha \to \Ac_0} & \leq \| I - \Pi_K \|_{\Ac_\alpha \to \Ac_0} \| \Lc\|_{\Ac_\alpha} \\ 
                &  \leq \left(1+\frac{2}{e^{2 \pi (\rho - \alpha)} - 1}\right) e^{- 2 \pi K \alpha} 
            \end{align*}
            having used that $\|\Lc\|_{\Ac_\alpha} \leq \| \Lc\|_{\Ac_\eta \to \Ac_\alpha}$. For the second RHS term of \eqref{eq:discEst3}, 
            \begin{align*}
               (*) = \| \Pi_K \Lc (I - \Pi_K) \|_{\Ac_\alpha \to \Ac_0} & \leq \| \Lc (I - \Pi_K)\|_{\Ac_\alpha \to \Ac_0}
            \end{align*}
            since $\| \Pi_K \|_{\Ac_0} = 1$ for all $K$, and so 
            \begin{align*}
                (*) & \leq \| \Lc\|_{\Ac_\eta \to \Ac_0} \| I - \Pi_K\|_{\Ac_\alpha \to \Ac_\eta} \\ 
                & \leq  \| \Lc\|_{\Ac_\eta \to \Ac_\alpha} \| I - \Pi_K\|_{\Ac_\alpha \to \Ac_\eta} 
                \\ 
                & \leq \left(1+\frac{2}{e^{2 \pi (\rho - \alpha)} - 1}\right) e^{-2 \pi (\alpha - \eta) K} 
            \end{align*}
            where in passing from the first to the second line we have used the estimate $\| \Lc\|_{\Ac_\eta \to \Ac_0} \leq \| \Lc \|_{\Ac_\eta \to \Ac_\alpha}$. 
            Combining estimates, the proof is now complete. 
        \end{proof}

        
        




\subsection{Estimating spectrum using the Householder argument}\label{subsec:householderExpanding3}



We now turn attention to the application of the developments in Section \ref{sec:householder} for the approximation of $\sigma(\Lc_T)$ by that of the finite-rank operators $\Lc_{T, K}$. This is summarized in the following result. 

\begin{proposition}\label{prop:householderForExpand3}
    Let Assumption \ref{ass:domainExpansion3} hold, and let $\alpha \in (\eta, \rho)$. Let $\mathfrak{r} > 0$ be fixed, and set 
   \begin{align} \label{eq:rValueAbs3} r = \left(\frac{1}{\mathfrak{r}}\right)^{\frac{\alpha}{\alpha-\eta}} \left(1+\frac{2}{e^{2 \pi (\rho - \alpha)} - 1}\right)^{\frac{\alpha}{\alpha-\eta}} \, . \end{align}
    Let $K \geq 1$ and let $\delta \geq r \| \Lc_T - \Lc_{T, K}\|_{\Ac_\alpha \to \Ac_0}$. 
    

	Let $\gamma_1, \dots, \gamma_k$ be a collection of simple closed loops with the following properties. Below, we write $\gamma_0$ for the circle of radius $\mathfrak{r}$ centered at $0$, so that $\Dc_{\gamma_0} = B_\mathfrak{r}(0)$. 
	\begin{itemize}
		\item[(1)] The sets $\overline{\Dc_{\gamma_i}}, 0 \leq i \leq k$ are disjoint from each other; and 
		\item[(2)] $\sigma_\delta(\Lc_{T, K}, \Ac_0) \subset \bigcup_{j = 0}^k \Dc_{\gamma_j}$. 
	\end{itemize}
	Then, 
	\begin{align}\label{eq:specContain3}
		\sigma(\Lc_T) \subset \bigcup_{j =0}^k \Dc_{\gamma_j} \, , 
	\end{align}
	and moreover, 
	\begin{align}\label{eq:multAgree3}
		\mathfrak{m}_{\Dc_{\gamma_j}}(\Lc_T) = \mathfrak{m}_{\Dc_{\gamma_j}}(\Lc_{T, K})
	\end{align}
	for all $1 \leq j \leq k$. 

    
\end{proposition}

    We caution that Proposition \ref{prop:householderForExpand3} does not follow immediately from Theorem \ref{thm:houseHolderSummary2}. The following points are relevant:
    \begin{itemize}
        \item[(a)] One can, in fact, apply Theorem \ref{thm:houseHolderSummary2} with the strong-weak pair $\Bc_s = \Ac_\alpha, \Bc_w = \Ac_\eta$, since Lemma \ref{lem:regImrpovement3} implies a DFLY inequality (Assumption \ref{ass:DFLY}) with this pairing. However, for computational reasons it is undesirable to compute resolvent norms in the exponentially-weighted space $\Ac_\eta$ due to potential underflow error, and far more stable to compute in $\Ac_0$. 
        \item[(b)] To prove Proposition \ref{prop:householderForExpand3}, it would hypothetically suffice to apply Theorem \ref{thm:houseHolderSummary2} with $\Bc_s = \Ac_\alpha, \Bc_w = \Ac_0$, but the authors are not aware of a suitable DFLY inequality for this strong-weak pair. Instead, we present below a proof of Proposition \ref{prop:householderForExpand3} which amounts to a mild reworking of the proof for Theorem \ref{thm:houseHolderSummary2}. 
    \end{itemize}
    




  \begin{proof}[Proof sketch for Proposition \ref{prop:householderForExpand3}]


    The proof of \eqref{eq:specContain3} is essentially a repeat of the analogous statement in Theorem \ref{thm:houseHolderSummary2} with $A = \Lc_T, \tilde{A} = \Lc_{T, K}$ and $\Bc_s = \Ac_\alpha, \Bc_w = \Ac_0$, on noting that for $\lambda \in \sigma(\Lc_T) \setminus \overline{B_\mathfrak{r}(0)}$, our assumptions imply $\delta \geq \gamma_{\Lc_T}(\lambda) \| \Lc_T - \Lc_{T, K}\|_{\Ac_\alpha \to \Ac_0}$ by Lemma \ref{lem:alphaToZeroEst3}. We emphasize that here that while $\Lc_{T, K} : \Ac_0 \to \Ac_0$ is bounded, hence $\Lc_{T, K} \in {\bf L}(\Ac_0, \Ac_\alpha)$, it is not clear that $\Lc_T : \Ac_0 \to \Ac_0$ is bounded. Nevertheless, Proposition \ref{prop:householder} can still be applied, since $\Lc_T$ is bounded on the strong space $\Ac_\alpha$, c.f. Remark \ref{rmk:weakAssumpHouseholder2}. 
    

    


    Turning now to \eqref{eq:multAgree3}, we pivot to the use of the strong-weak pair $\Bc_s = \Ac_\alpha, \Bc_w = \Ac_\eta$ and consider the path of operators 
        \[A_t = t \Lc_T + (1 - t) \Lc_{T, K} \, , \quad t \in [0,1] \,.  \]
    We highlight the following points: 
    \begin{itemize}
        \item[(I)] Each of $\Lc_T, \Lc_{T, K}$ satisfy a version of the DFLY inequality (Assumption \ref{ass:DFLY}) with $C_1 = 1, C_2 = 0$ and $\beta = \|  \Lc_T \|_{\Ac_\eta \to \Ac_\alpha}$ (Lemma \ref{lem:regImrpovement3}), hence the same can be said for all $A_t, t \in [0,1]$ with the same constants. Moreover, $\| A_t\|_{\Ac_\eta}$ is uniformly bounded over $t \in [0,1]$. 
        \item[(II)] Repeating the argument for \eqref{eq:specContain3}, it holds that 
        \[\sigma(A_t) \subset \sigma_\delta(\Lc_{T, K}, \Ac_0)\subset \bigcup_{j =0}^k \Dc_{\gamma_j} \quad \text{for all $t \in [0,1]$.}\] 
        \item[(III)] In view of the identity $A_t - A_{t'} = (t - t') (\Lc_T - \Lc_{T, K})$, it follows that $t \mapsto A_t$ is a continuous path of operators in the strong-to-weak mixed norm $\| \cdot \|_{\Ac_\alpha \to \Ac_\eta}$. 
    \end{itemize}
    Items (I) -- (III) are precisely the main assumptions in the spectral continuity argument in the proof of Theorem \ref{thm:houseHolderSummary2}, c.f. Lemma \ref{lem:DFLYSpecCty2}. Repeating the argument given there implies \eqref{eq:multAgree3}. 
  \end{proof}

  \begin{remark}
    The basic idea of the proof of Proposition \ref{prop:householderForExpand3} is to combine (i) containment of $\sigma(A_t)$ inside $\cup_j \Dc_{\gamma_j}$ for all $t \in [0,1]$, and (ii) some continuity argument for $\sigma(A_t)$ as $t$ varies. By the time (i) is established, it does not matter whether the pseudospectrum used there came from the weak space $\Bc_w$ or some other norm. 
  \end{remark}

%% file: examples.tex
\section{Results for explicit analytic expanding maps}\label{sec:examples}



We present now an application to two explicit analytic expanding maps of the circle: a Blaschke product (Section \ref{subsec:blaschke4}) and a perturbed doubling map (Section \ref{subsec:pertDoub4}). 
The spectrum of transfer operators for uniformly expanding Blaschke products is explicitly known \cite{slipantschuk2013analytic,bandtlow2017spectral}, and their inclusion in Section \ref{subsec:blaschke4} serves as a benchmark of our method. That for the perturbed doubling map in Section \ref{subsec:pertDoub4} is not explicitly known, and for this example our techniques provide new, mathematically rigorous information. 

The scripts for the experiments and the notebooks used to estimate the constants can be accessed at 
\begin{center}
\url{https://github.com/orkolorko/ExperimentsPseudospectra.jl}
\end{center}

The script can run on a local machine or on a SLURM cluster; \texttt{sbatch} files are available to schedule 
the job on the cluster. 

\subsection{A Blaschke product}\label{subsec:blaschke4}

{
\begin{definition}
	\emph{Blaschke products} are rational maps of the form  
		\begin{align} \label{eq:absBlaschke4}T(z) = C \prod_{i = 1}^n \frac{z - a_i}{1 - \bar{a_i} z} \, ,\end{align}
		where $a_1, \dots, a_n \in \C$ satisfy $|a_i| < 1$ for all $i$, and where $|C| =1$. For example, the usual doubling map $T(z) = z^2$ is a Blaschke product with $n = 2, a_1 = a_2 = 0$ and $C = 1$. 
\end{definition}
	
	 From the definitions, $T$ has poles at $\{1 / \bar a_i\}$ and zeros at $\{ a_i\}$; leaves the unit circle $S^1$ invariant; and is holomorphic on a neighborhood of the closed unit disk. 
	The following results address when $T$ is uniformly expanding along $S^1$. 
	\begin{proposition}\label{prop:blaschkeProperties4} \
		\begin{itemize}
			\item[(a)] A Blaschke product $T$ as in \eqref{eq:absBlaschke4} is uniformly expanding on $S^1$ if \[\sum_1^n \frac{1 - |a_i|}{1 + |a_i|} > 1 \,. \]  
			\item[(b)] A Blaschke product $T$ is uniformly expanding on $S^1$ if and only if it has a unique fixed-point $z_0$ in the open unit disk. 
		\end{itemize}
	\end{proposition}
	Item (a) is \cite[Corollary to Proposition 1]{martin1983finite}, and item (b) is a combination of \cite[Theorem 1]{tischler2000blaschke} and \cite[Proposition 2.1]{pujals2006expanding}.

	 Assumption \ref{ass:domainExpansion3} holds for some $0 < \eta < \rho$ (\cite[Lemma 2.2]{slipantschuk2013analytic}), and so by Corollary \ref{cor:compactnessTfer3}, the transfer operator $\Lc_T$ of a uniformly expanding Blaschke products is compact $\Ac_\alpha \to \Ac_\alpha$ for all $\alpha \in [\eta, \rho)$, with spectrum $\sigma(\Lc_T) = \sigma(\Lc_T, \Ac_\alpha)$ independent of $\alpha$ in that range. 
	
	\begin{theorem}[\cite{bandtlow2017spectral}]\label{thm:explicitBlaschkeSpectra4}
		Let $T$ be a uniformly expanding Blaschke product, and let $z_0$ be its unique fixed point in the open unit disk. Let $\mu = T'(z_0)$. Then, 
		\[\sigma(\Lc_T) = \{ 0, 1\} \cup \{ \mu^n : n \geq 1\} \cup \{ \bar \mu^n : n \geq 1\} \,. \]
	\end{theorem}
	\begin{remark}\label{rmk:spectrum}
		The results of \cite{bandtlow2017spectral} are stated for spectrum of $\Lc_T$ on the Hardy space $H^2(U)$ for a class of appropriately chosen annuli $U\subset \C$ containing the unit circle, but it is straightforward to show that the $H^2(U)$-spectrum, $U = \mathfrak{U}_{2 \pi \alpha}$ is the same as the $\Ac_\alpha$ spectrum under Assumption \ref{ass:domainExpansion3}, c.f. \cite[Section 3]{bandtlow2017spectral}. 
	\end{remark}


\subsubsection*{Computer-assisted results for a specific Blaschke product}

	The specific map to which our technique is applied is the degree-two Blaschke product  
	\begin{align} \label{eq:explicitBlaschke4} T(z) = \frac{z(\mu - z)}{1 - \bar \mu z} \, , \end{align}
	which is of the form \eqref{eq:absBlaschke4} with $C = -1$, $a_1 = 0, a_2 = \mu$ for $\mu = r e^{i \phi}$ with $r = 3\sqrt{2}/8$ and $\phi = \pi / 8$. 
	By Proposition \ref{prop:blaschkeProperties4}(a), $T$ is uniformly expanding, and its fixed point in the open unit disk is $z_0 = 0$ with $T'(z_0) = \mu$. 
	
The matrix, the outputs and the certification logs for this experiment are available 
as a Dataset on the Harvard Dataverse \cite{DVN/Blaschke}.

\begin{lemma}\label{lem:blaschkeDomainExpCAP}
The map $T$ as in \eqref{eq:explicitBlaschke4} satisfies Assumption \ref{ass:domainExpansion3} with 
$\rho = 0.583052$ and $\eta = 0.49149149$. 
\end{lemma}
\begin{proof}[Proof comments]
This is proved by a two-step approach, presented in the notebook \\
\begin{center}
\texttt{ExperimentsPseudospectra.jl/notebook/Blashke/Blaschke.ipynb}	
\end{center}
First, a numerical, nonrigorous search is performed to find suitable optimal candidates $\eta$, $\rho$ and $\alpha$. Second, a lower bound on the value of $\rho$ is certified by enclosing the image of an annulus by using interval arithmetic.
\end{proof}

Theorem \ref{thm:explicitBlaschkeSpectra4} and Remark \ref{rmk:spectrum} then implies \[\sigma(\Lc_T) = \{ 0, 1\} \cup \{ r^n e^{i n \phi} : n \geq 1\} \cup \{ r^n e^{-i n \phi} : n \geq 1\}  \] 
which is consistent with the following result derived from our method } 

\begin{theorem}\label{thm:CAPBlaschke4}
	Let $T$ be the Blaschke product as in \eqref{eq:explicitBlaschke4}. Then, 
	\begin{itemize}
		\item[(a)] $\sigma(\Lc_T) \subset \bigcup_0^3 F_i$; and 
		\item[(b)] $\mathfrak{m}_{F_i}(\Lc_T) = 1$ for $i = 1, 2, 3$; 
	\end{itemize}
	where $F_0, F_1, F_2, F_3$ are the open disks with centers and radii given in Table \ref{table:blaschkeProd}. 
\end{theorem}

\begin{table}[h]
	\centering
	\begin{tabular}{c|c|c|c}
		  & Center & Radius & $\sup r_{\Lc_{T, K}}(z)$ \\ \hline
	$F_0$ &    $0$    &  $0.51$  &  $3.93 \cdot 10^6$ \\ \hline
	$F_1$ &    $1$    &  $0.1$  &  $797.15$  \\ \hline
	$F_2$ &    $0.4899611118286257 - i 0.20294853755482672$    & $0.01$   &  $3.11\cdot 10^{6}$  \\ \hline
	$F_3$ &    $0.4899611118286279 + i 0.20294853755482584$    &  $0.01$  &  $3.11\cdot 10^{6}$  \\ 
	\end{tabular}
	\caption{Specification of disks $F_1, \dots, F_4$ in Theorem \ref{thm:CAPBlaschke4}}
	\label{table:blaschkeProd}
\end{table}

\begin{proof}[Comments on Proof of Theorem \ref{thm:CAPBlaschke4}]
The value $K = 128$ fixed throughout.
		The value $\mathfrak{r}$, the radius of disk $F_0$, is set according to Table \ref{table:blaschkeProd}. The values of $\eta, \rho$ are set as in Lemma \ref{lem:blaschkeDomainExpCAP}. We fix $\alpha = 0.5758488557738615$. With these choices, the value $r = 2.21\cdot 10^{14}$ is a certified upper bound for the RHS of \eqref{eq:rValueAbs3}, allowing to set $\delta^{-1} = 2.99\cdot 10^{13}$ (using the bounds in Lemma \ref{lem:Galerkinapprox3} to control $\| \Lc_T - \Lc_{T, K} \|_{\Ac_\alpha \to \Ac_0}$).
		
		As to the application of Proposition \ref{prop:householderForExpand3}: the disks $F_i$ defined in Table \ref{table:blaschkeProd} are clearly disjoint. It remains to show $\sigma_\delta(\Lc_{T, K} , \Ac_0) \subset \cup_i F_i$. 
		This, in turn,  is checked using Lemma \ref{lem:subharmonic2}. That $\sigma(\Lc_{T, K}) \subset \cup_0^3 F_i$ follows from a rigorously-approximated Schur decomposition of $\Lc_{T, K}$. 
		Bounds in the last column of Table \ref{table:blaschkeProd} are upper bounds for $\sup_{z \in \gamma_i} r_{\Lc_{T, K}}(z, \Ac_0)$, where $\gamma_i = \partial F_i, 1 \leq i \leq 4$. Since these upper bounds are all $\leq \delta^{-1}$, follows that 
		$\sigma_\delta(\Lc_{T, K}, \Ac_0) \subset \cup_i F_i$ as desired. 

		We note that the truncation $\Lc_{T, K}$ is rigorously approximated using an interval arithmetic implementation of the fast Fourier transform with $1048576$ points -- see Section \ref{subsec:genComments4} for details. 
		Further implementation details on the Schur decomposition and resolvent upper bounds are given in Section \ref{sec:CAP5}. 
\end{proof}



\begin{figure}
     \centering
     \begin{subfigure}{0.45\textwidth}
         \centering
         \includegraphics*[width=60mm]{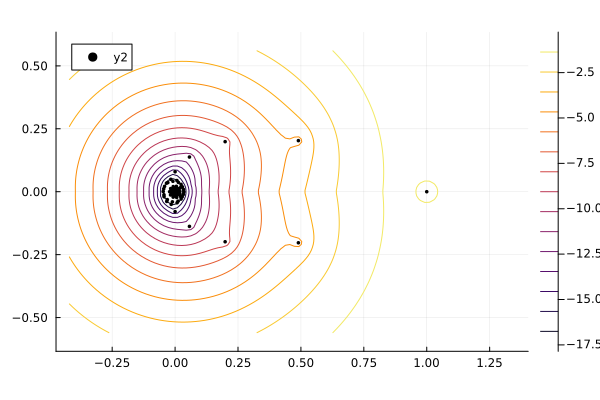}
		\caption{Non validated $
		\ell^2$-pseudospectra of the Blaschke product \eqref{eq:absBlaschke4}}
		\label{fig:pseudospectraBlashke}
     \end{subfigure}
     \begin{subfigure}{0.45\textwidth}
         \centering
         \includegraphics[width=\linewidth]{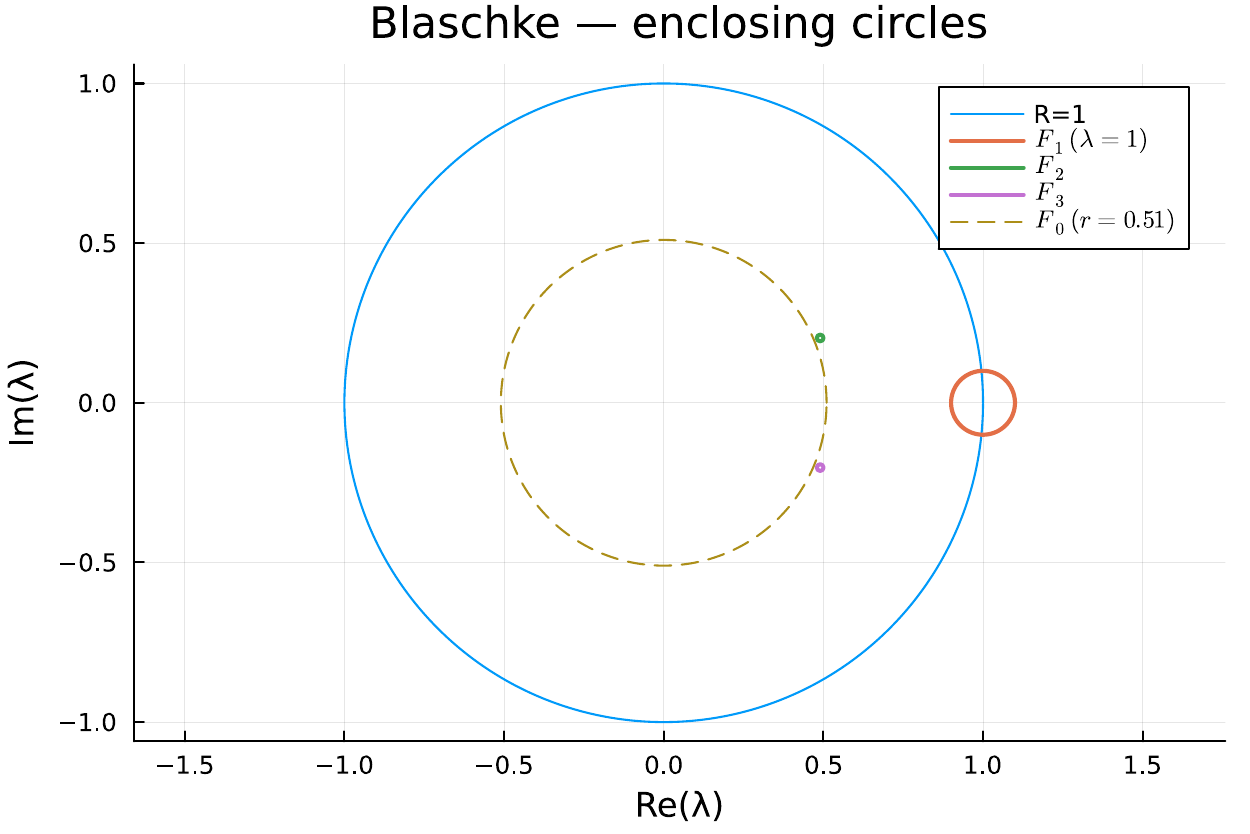}
         \caption{Computed enclosures}
         \label{fig:1b}
     \end{subfigure}
     \caption{Circle enclosure for the Blaschke product}
     \label{fig:circle_enc_Blashke}
\end{figure}

\subsection{Perturbed doubling map}\label{subsec:pertDoub4}

The second example we exhibit is a large perturbation of the doubling map

\begin{align} \label{eq:pertDoublingMap4} T(z) = i z^2 \exp \left[{\left(\frac12 - b \pi\right)(z - z^{-1})}\right]\end{align}
with $b=5/64+1/128+1/256$.
which is easily seen to be uniformly expanding, with 
\[|T'(x)| \geq 1 + \frac{\pi}{8} > 1.\]

The matrix, the outputs and the certification logs for this experiment are available 
as a Dataset on the Harvard Dataverse \cite{DVN/Doubling}.

\begin{lemma} \label{lem:domainExpPertDoubMap4}
The map $T$ as in \eqref{eq:pertDoublingMap4} satisfies Assumption \ref{ass:domainExpansion3} with 
$\rho = 0.312891$ and $\eta = 0.22211055$.
\end{lemma}


\begin{figure}
     \centering
     \begin{subfigure}{0.45\textwidth}
         \centering
         \includegraphics*[width=60mm]{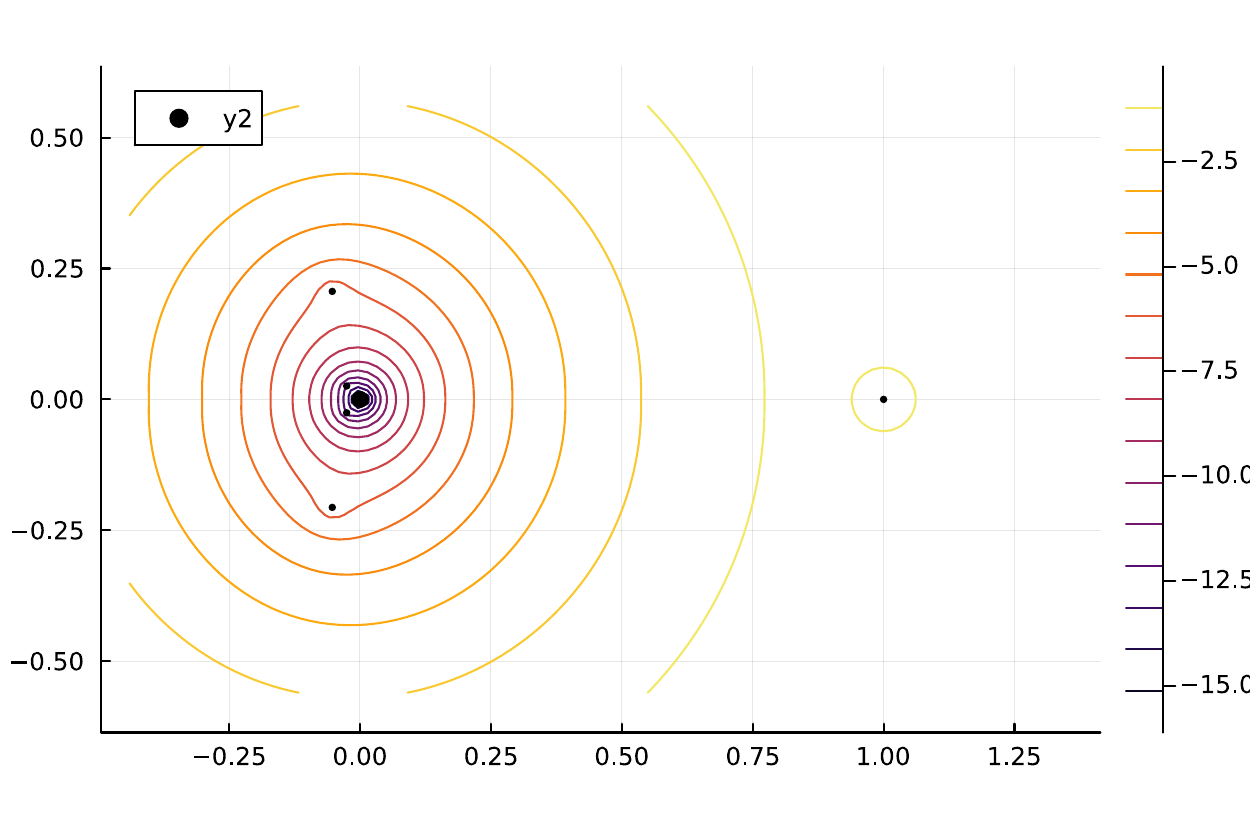}
		\caption{Non validated $
		\ell^2$-pseudospectra of the perturbed doubling map product \eqref{eq:pertDoublingMap4}}
		\label{fig:pseudospectraArnold}
     \end{subfigure}
     \begin{subfigure}{0.45\textwidth}
         \centering
         \includegraphics[width=\linewidth]{circles_Arnold.pdf}
         \caption{Computed enclosures, refer to Figure \ref{fig:circle_enc_Arnold_intro} for a zoom}
         \label{fig:circe_enc_Arnold_far}
     \end{subfigure}
     \caption{Circle enclosure for the Blaschke product}
     \label{fig:circle_enc_Arnold}
\end{figure}

\begin{theorem} \label{thm:pertDoubleMap4}
	Let $T$ be as in \eqref{eq:pertDoublingMap4}. Then,
	\begin{itemize}
		\item[(a)] $\sigma(\Lc_T) \subset \cup_0^3 F_i$; and 
		\item[(b)] $\mathfrak{m}_{F_i}(\Lc_T) = 1$ for $i = 1,2,3$; 
	\end{itemize}
	where $F_0, F_1, F_2, F_3$ are the open disks with centers and radii given in Table \ref{table:perturbDoublingMap}. 
	\end{theorem}
	\begin{table}[h]
		\centering
		\begin{tabular}{c|c|c|c}
			  & Center & Radius & $\sup r_{\Lc_{T, K}}(z)$ \\ \hline 
		$F_0$ &    $0$    &  $0.21$  & $1.46\cdot 10^9$  \\ \hline
		$F_1$ &    $1$    &  $0.1$  & $641.76$  \\ \hline
		$F_2$ &    \edited{$-0.052863658851127746 +i 0.2062498438069236$}   &  $0.001$  & \edited{$4.486\cdot 10^{9}$} \\ \hline
		$F_3$ &    \edited{$-0.05286365885120539 - i 0.20624984380695763$}    &  $0.001$  & \edited{$4.487\cdot 10^{9}$}    \\
		\end{tabular}
		\caption{Specification of disks $F_1, \dots, F_4$ in Theorem \ref{thm:pertDoubleMap4}.}
		\label{table:perturbDoublingMap}
	\end{table}

	\begin{proof}[Comments on Proof of Theorem \ref{thm:pertDoubleMap4}] 
		With $K = 128$ fixed, the values of $\eta, \rho$ are as in Lemma \ref{lem:domainExpPertDoubMap4}. Fixing $\alpha=0.308389$, the value $r = 1.783\cdot 10^{9}$ is an upper bound on the RHS of \eqref{eq:rValueAbs3}, allowing to use $\delta^{-1} = 1.04 \cdot 10^{19}$. Additional details are as in the proof of Theorem \ref{thm:explicitBlaschkeSpectra4}.
	\end{proof}


\begin{remark}\label{rmk:choiceOfT4}
The	map $T$ in \eqref{eq:pertDoublingMap4} was selected \edited{based} on numerical evidence that weakened expansion near $z = -1$ combined with $90{}^\circ$-rotation creates a resonance pair $\lambda, \bar \lambda$ relatively far from $0$, which makes comparatively easier to isolate this spectrum with our methods. 
\end{remark}

\subsection{General comments on the computer aided proofs}\label{subsec:genComments4}

\edited{In the next Section \ref{sec:CAP5}, we will document a procedure for 
rigorously estimating pseudospectra of finite-dimensional matrices.
In preparation, in this subsection we briefly address how to bridge between the \emph{operator} resolvent estimates in Proposition \ref{prop:householderForExpand3} and the matrix resolvent estimates. 

\subsubsection*{Reduction of resolvent estimates to finite-dimensional linear algebra.}
To start, we distinguish between the resolvents
\[r_{\Lc_{T, K}}(z; \Ac_0) = \edited{\| (z - \Lc_{T, K})^{-1}\|_{\Ac_0}}
\]
of the finite-rank operators $\Lc_{T, K}$ acting on $\Ac_0$ and the 
associated linear operator $\Lc_{T, K} : \Pi_K \Ac_0 \to \Pi_K \Ac_0$ of a finite-dimensional space. The former is what we need in Proposition \ref{prop:householderForExpand3}, while the latter is what we can actually estimate in Section \ref{sec:CAP5}. 
Thankfully, the following allows to pass from one to the other easily. 
\begin{lemma}\label{lem:finDimResolvents4}
	For any $z \notin \sigma(\Lc_{T, K})$, it holds that 
	\[\max\{ |z|^{-1}, \| (z - \Lc_{T, K})^{-1} \|_{\Pi_K \Ac_0}\} \geq \| (z - \Lc_{T, K})^{-1} \|_{\Ac_0} \,. \]	
	In particular, for any $\delta > 0$ it holds that
	\[\sigma_\delta(\Lc_{T, K}, \Ac_0) \cap B_\delta(0)^c \subset \sigma_\delta(\Lc_{T, K}, \Pi_K \Ac_0) \,. \]
\end{lemma}
Here, $\Pi_K \Ac_0$ refers to the operator norm on the (finite-dimensional) range of $\Pi_K$ induced by the $\Ac_0$ norm.
Lemma \ref{lem:finDimResolvents4} is exactly what we need: pseudospectrum of $\Lc_{T, K}$ regarded on $\Ac_0$ is contained in that of $\Lc_{T, K}$ regarded on $\operatorname{Range}(\Pi_K)$. 
This comes at the cost of avoiding point spectrum the open ball $B_\delta(0)$, which is already done for other reasons elsewhere; in practice, the restriction to $B_\delta(0)^c$ in Lemma \ref{lem:finDimResolvents4} changes nothing. 
\begin{remark}
	It is natural that Lemma \ref{lem:finDimResolvents4} applies only to spectrum away from $0$, since passing from $\Pi_K \Ac_0$ to $\Ac_0$ incurs introducing an infinite-dimensional eigenspace at $z = 0$, completely avoided by $\Lc_{T, K}$ regarded on $\Pi_K \Ac_0$. 
\end{remark}

\begin{proof}[Proof of Lemma \ref{lem:finDimResolvents4}]
	Given $f \in \Ac_0$, let $f = f_K + f_K^\perp$ where $f_K = \Pi_K f$, so that $f_K^\perp$ is spanned by Fourier modes with wavenumber $|k| > K$. Since $\Lc_{T, K}$ annihilates $f_K^\perp$, we have 
	\begin{align*}
		(z - \Lc_{T, K}) f = (z - \Lc_{T, K}) f_K + z f_K^\perp \,. 
	\end{align*}
	Since each term on the RHS comes from nonoverlapping sets of Fourier modes, the form of the $\Ac_0$ norm ensures
	\begin{align*}
		\| (z - \Lc_{T, K}) f \|_{\Ac_0} \geq \min\{ |z|, \| (z - \Lc_{T, K}) \|_{\Ac_0}\} \| f \|_{\Ac_0} \,,
	\end{align*}
	from which the desired statement follows.
\end{proof}

}
\subsubsection*{Computation of Galerkin approximations}
Let $T$ be as in Section \ref{subsec:blaschke4} or \ref{subsec:pertDoub4}. 
We discuss how we compute the Galerkin approximation of the transfer operator.

This follows from some simple observations: the Fourier transform is an isometry 
$\operatorname{Range}(\Pi_K) \cong \C^n$ with  $n = 2 K + 1$. 
It carries  $\Lc_{T, K}$ to Fourier matrix $M = M_{T, K} \in \C^n$ with entries
\[M_{ij} = \frac{1}{2 \pi i} \int_{S^1} \Lc_{T, K}[e_{\hat j}] e_{- \hat i} \frac{dz}{z} \, , \]
where $e_m(z) = z^m, m \in \Z$, and where $\hat i = i - (K + 1), \hat j = j - (K+1)$ for $1 \leq i, j \leq n = 2 K + 1$. 

The Fourier transform takes \edited{the} $\Ac_0$-norm on $\operatorname{Range}(\Pi_K)$ to the $\ell^1$ norm on $\C^n$. 
Computing and certifying the $\ell^1$ norm of the inverse of a matrix is a somewhat delicate
problem, since it involves certifying its inverse and this is an ill conditioned problem when 
a matrix is almost singular. 
We decided to work with $\ell^2$ norm instead, for which there is an established theory of validated certification 
of the norm of the inverse \cite{rump2010verified}, appealing to the inequalities
\[\|v\|_{\ell^1} \leq \|v\|_{\ell^2} \leq \sqrt{n} \|v\|_{\ell^1}\] 
so that, for an $n\times n$ matrix $B$ one has 
\[\|B\|_{\ell^1} \leq \sqrt{n} \|B\|_{\ell^2} \,. \]

The Fourier matrix itself was computed through the use of the Cooley-Tuckey FFT algorithm
with validated numerical error bounds; due to our explicit estimates on the decay rates of the functions 
in the analytical functional spaces, the aliasing error can be explicitly bounded.
The Galerkin approximation in both our examples truncates frequencies $|k| \geq 128$, and the FFT is computed on $1048576$ points.




%% file: validated.tex
\section{Validated numerical methods for pseudospectra} \label{sec:CAP5}

The goal of this section is to lay out an approach for rigorously estimating resolvent norms of possibly large matrices, as will be needed in the computer-assisted steps in the proofs of Theorems \ref{thm:CAPBlaschke4} and \ref{thm:pertDoubleMap4}. 

Throughout, $n \geq 1$ is fixed and $\|\cdot\|$ refers to the matrix norm induced by the usual Euclidean inner product on $\C^n$. Given $M \in M_{n \times n}(\C)$ we write $r_M (z) = \|(z - M)^{-1}\|$ for the resolvent in the Euclidean norm. Write $M^* = \bar M^T$ for Hermitian dual of $M$. $\sigma_\delta$ always refers to $\| \cdot \|$-pseudospectrum.

\subsubsection*{Overview of the computational problem}

\newcommand{\Tc}{T}

We begin with a nonrigorous overview of the basic computational steps involved. 

Given $M \in M_{n \times n}(\C)$ and $\delta > 0$, we seek the following: 
a collection of simple closed curves $\gamma_1, \dots, \gamma_k \subset \C$ (e.g., circles) enclosing open regions $\mathcal{D}_{\gamma_i}$ with disjoint closures such that 
\begin{align} \label{eq:pseudospecContainmen5}\sigma_\delta(M) \subset \cup_i \mathcal{D}_{\gamma_i} \,. \end{align}
In view of subharmonicity of resolvents as in Lemma \ref{lem:subharmonic2}, equation \eqref{eq:pseudospecContainmen5} follows on checking 
\begin{itemize}
\item[(a)]  $\sigma(M) \subset \bigcup_i \mathcal{D}_{\gamma_i}$; and 
\item[(b)] $r_M(z) \leq \delta^{-1}$ for $z \in \gamma_i, 1 \leq i \leq k$. 
\end{itemize}
For this, ignoring numerical error for the time being, one could take the following steps: 
\begin{itemize}
	\item[(i)] Form a Schur decomposition $M = Z \Tc Z^*$, where $Z \in U(n)$ (i.e., $Z$ is unitary) and $T$ is upper triangular. Then, the spectrum $\sigma(M)$ coincides with $\sigma(\Tc)$, and consists of the diagonal elements $\lambda_1, \dots, \lambda_n$ of $\Tc$.
	\item[(ii)] Group together $\lambda_1, \dots, \lambda_n$ into nearby ``batches'' enclosed by simple closed curves $\gamma_i$ for which the regions $\mathcal{D}_{\gamma_i}$ are disjoint. This ensures that item (a) above is now satisfied. 
	\item[(iii)] Since $r_\Tc(z) = r_M(z)$, to satisfy item (b) above it suffices to estimate $\sup_{z \in \cup_i \gamma_i} r_\Tc(z)$. Fix a sufficiently-fine finite mesh $\xi_i \subset \gamma_i$ for each $i$, and compute a singular value decomposition of $z - T$ for each $z \in \xi_i$, using the identity:
	\begin{align} \label{eq:lastSingularValue5}r_\Tc (z) = \|(z - \Tc)^{-1}\| = {1 \over \sigma_n(z - T)} \,. \end{align}
	
\end{itemize}

Above, for $B \in M_{n \times n}(\C)$ we write 
\[\| B \| = \sigma_1(B) \geq \sigma_2(B) \geq \dots \geq \sigma_n(B) = \| B^{-1}\|^{-1}\] for the singular values of $B$. 

To turn the above into a proof will require rigorous error estimation to overcome inexactness of floating-point arithmetic and the various decompositions we obtain. These issues are addressed in the rest of this section. 

The plan is as follows. In Section \ref{subsec:Schur}, we address how to estimate $r_T(z)$ from $r_M(z)$, in view of the inexactness of (i) the matrix $M$ itself and (ii) that of a numerically-derived Schur decomposition. 
The main result is Proposition \ref{prop:approxSchurToHouseholder5}, which addresses how estimates of $r_T(z)$ leads to enclosures for $\sigma_\delta(M)$. In Section \ref{subsec:SVD5} we recall details of a quantitative ``approximate'' Singular Value Decomposition due to Rump, and address the finite meshing procedure along the curves $\gamma_i$. Finally, Section \ref{subsec:implementation5} comments on algorithms and concrete implementation details in the Julia package \texttt{BallArithmetic.jl} \cite{BallArithmetic.jl}. 

\subsection{Approximate Schur decomposition}\label{subsec:Schur}

Let $M \in M_{n \times n}(\C)$ and suppose
\begin{align} \label{eq:approxSchur5}
	M = Z T Z^* + E_M \, , 
\end{align}
where $T \in M_{n \times n}(\C)$ is upper triangular, and $E_M \in M_{n \times n}(\C)$ is a (hopefully small) error term. Here, the matrix $Z \in M_{n \times n}(\C)$ will be assumed to be \emph{approximately unitary}, i.e.,
\begin{align}\label{eq:approxUnitary5}
	I = Z Z^* + E_Z
\end{align}
where $E_Z \in M_{n \times n}(\C)$ is another (hopefully small) error term.

The following compares $r_M(z)$ and $r_T(z)$
in terms of explicit bounds on $E_M$ and $E_Z$. The following estimates have been made somewhat nonoptimally in favor of readability. 

\begin{lemma}\label{prop:approxSchur5}
	Let $\epsilon > 0$. Let $M, Z, T, E_M, E_Z \in M_{n \times n}(\C)$ be as in \eqref{eq:approxSchur5} and \eqref{eq:approxUnitary5}, and assume 
	\begin{gather*}
		\|E_M\|, \|E_Z\| \leq \epsilon \, , \qquad \text{ and } \qquad 
		\|Z\|, \|Z^{-1}\| \leq 1 + \epsilon \,. 
	\end{gather*}
	Finally, suppose $z \in \C \setminus \sigma(T)$ is such that 
	{ \begin{align}\label{eq:boundMaxrTz5}
		2 \epsilon(1 + \epsilon)^2 r_T(z) \cdot \max\{ 1 , |z| \} < \frac12 \,. 
	\end{align} }
	 Then, 
	{\begin{align}\label{eq:schurUpperBound}
		r_M(z) \leq \frac{2 (1 + \e)^2 r_T(z)}{1 - 2 \epsilon (1 + \e)^2 r_T(z)} \,. 
	\end{align}  } 
\end{lemma}

We now turn to the main problem: given $T$ and $Z$, to enclose the pseudospectrum of a matrix $M$ close to $ZTZ^*$. 
The following proposition summarizes the discussion in this section and presents it in a form adapted to 
Theorem \ref{thm:houseHolderSummary2} and Proposition \ref{prop:householderForExpand3}.

\begin{proposition}\label{prop:approxSchurToHouseholder5}
	Let $T$ be an upper triangular $n \times n$ matrix and let $Z \in M_{n \times n}(\C)$ be such that 
	\[\| E_Z\| \leq \epsilon, \qquad \| Z\|, \| Z^{-1}\| \leq 1 + \epsilon \]
	where $E_Z$ is as in \eqref{eq:approxUnitary5} and where $\epsilon \in (0,1)$. Let $C_0 \geq 1$ and suppose $\| T\| \leq C_0$. Fix 
	\begin{align}\label{eq:condDelta05} \delta_0 \geq 4 \epsilon \big(\epsilon + C_0 (1 + \epsilon)^2\big) (1 + \epsilon)^2 \,. \end{align}

	Next, let $\gamma_1, \dots, \gamma_k \subset \C$ be simple closed curves bounding open regions $\mathcal{D}_{\gamma_i}$ such that the closures $\overline{\mathcal{D}_{\gamma_i}}$ are pairwise disjoint, $1 \leq i \leq k$. 
	Assume 
	\begin{itemize}
		\item[(i)] $\sigma(T) \subset \bigcup_i \mathcal{D}_{\gamma_i}$; and 
		\item[(ii)] $r_T(z) \leq \delta_0^{-1}$ for all $z \in \bigcup_i \gamma_i$. 
	\end{itemize}
	
	Then, for any $M \in M_{n \times n}(\C)$ with $\| M - Z T Z^*\| \leq \epsilon$, it holds that 		
	\begin{align} \label{eq:containPSM5}\sigma(M) \subset \bigcup_i \mathcal{D}_{\gamma_i} \end{align}
	and 
	\begin{align} \label{eq:MresolvUpper5} r_M(z) \leq 4 (1 + \epsilon)^2 \delta_0^{-1} \quad \text{ for all } z \in \bigcup_i \gamma_i \,. \end{align}
	In particular, it follows from Lemma \ref{lem:subharmonic2} that 
	\[\sigma_\delta(M) \subset \bigcup_i \mathcal{D}_{\gamma_i} \]
	for all $\delta \leq \delta_0 / 4 (1 + \epsilon)^2 $. 
\end{proposition}  

\subsubsection*{Proofs}

Below, we make prolific use of the following classical resolvent estimate.
\begin{lemma}\label{lem:classicalResolv5}
	Let $B, E \in M_{n \times n}(\C)$ and suppose $z \in \C \setminus \sigma(B)$ is such that $\|E\| < r_B(z)^{-1}$. Then, 
	\begin{align}\label{eq:classicalResolventEst}
		r_{B + E}(z) 
		\leq  \| (I - (z-B)^{-1} E)^{-1}\| r_B(z)
		\leq \frac{r_B(z)}{1 - \|E\| \cdot r_B(z)} 
	\end{align}
\end{lemma}
\begin{proof}[Proof of Lemma \ref{lem:classicalResolv5}]
	By direct computation, 
	\[
		\|(z-(B + E))^{-1}\| = \| (I-(z-B)^{-1}E)^{-1}(z-B)^{-1}\| \, , 
	\]
	from which \eqref{eq:classicalResolventEst} immediately follows. 
\end{proof}

In the proof below, we write $Z^{-*}$ for the dual-inverse of $Z$, i.e., $Z^{-*} = (Z^{-1})^* = (Z^*)^{-1}$. 

\begin{proof}[Proof of Lemma \ref{prop:approxSchur5}]
	We start with the upper bound on $r_M(z)$. Equation \eqref{eq:classicalResolventEst} implies
	\begin{align}
		r_M(z) & \leq \frac{r_{Z T Z^*}(z)}{1 - \e \cdot r_{Z T Z^*}(z)} \, , 
	\end{align}
	provided $\epsilon r_{Z T Z^*}(z) < 1$, which will be checked shortly. For $r_{ZTZ^*}(z)$, we compute
	\begin{align*}
		r_{Z T Z^*}(z) &= \|(Z (z - T) Z^* + z(I - Z Z^*))^{-1} \|
		\\ & = \|Z^{-*} ((z - T) + z Z^{-1} (I - Z Z^* ) Z^{-*} )^{-1} Z^{-1}\| \\ &  \leq (1 + \epsilon)^2 \|((z - T) + z Z^{-1} E_Z Z^{-*} )^{-1} \| \, ,  
	\end{align*}
	hence by Lemma \ref{lem:classicalResolv5}, 
	\begin{align}\label{eq:boundZTZ5}
		\begin{aligned}
			r_{Z T Z^*}(z)  
						   & \leq \frac{(1 + \e)^2 r_T(z)}{1 - \e |z| (1 + \e)^2 r_T(z)}                 \\
						   & \leq 2 (1 + \e)^2 r_T(z) 
		\end{aligned}
	\end{align}
	assuming $\e |z| (1 + \e)^2 r_T(z)\leq \frac12$, which holds from \eqref{eq:boundMaxrTz5}. 	Combining with our previous estimate completes the proof of the upper bound on $r_M(z)$, noting that $\epsilon r_{Z T Z^*}(z) < 1$ is implied by \eqref{eq:boundMaxrTz5} and \eqref{eq:boundZTZ5}. 
\end{proof}

{\begin{proof}[Proof of Proposition \ref{prop:approxSchurToHouseholder5}]
  We begin by noting that \eqref{eq:MresolvUpper5} is immediate from the upper bound on $r_M(z)$ coming from Lemma \ref{prop:approxSchur5}, provided 
  \[2 \epsilon (1 + \epsilon)^2 \delta_0^{-1} \leq \frac12 \, , \]
  which follows from \eqref{eq:condDelta05} since $C_0 \geq 1$.
  
  It remains to prove \eqref{eq:containPSM5}. To start, the 
	  usual Householder Theorem and the estimate $\| M - Z T Z^*\| \leq \epsilon$ imply 
	  \[\sigma(M) \subset \sigma_\epsilon(ZTZ^*) \,, \]
	  and so it will be enough to check that 
	  \[\sigma_\epsilon(Z T Z^*) \subset \sigma_{\delta_0}(T) \, .\]
	  We will need the following. 
	  \begin{claim}\label{cla:containZTZ5}
		  \[\sigma_\epsilon(Z T Z^*) \subset \overline{B_R(0)} \, ,\]
		  where $R = \epsilon + C_0 (1 + \epsilon)^2$. 
	  \end{claim}
	  \begin{proof}[Proof of Claim]
		  Let $|z| > R$. Then, 
		  \[r_{ZTZ^*} (z) < \frac{1}{R - \| Z T Z^*\|} \leq \frac{1}{R - C_0 (1 + \epsilon)^2} = \frac{1}{\epsilon} \, , \]	
		  hence $z \notin \sigma_\epsilon(ZTZ^*)$. 
	  \end{proof}
  
	  To complete the proof of Proposition \ref{prop:approxSchurToHouseholder5}, let $z \in \sigma_\epsilon(Z T Z^*)$. If $z \in \sigma(T)$ there is nothing to check, and so we proceed with assuming $z \notin \sigma(T)$. Next, observe that equation \eqref{eq:boundZTZ5} and the bound $|z| \leq R$ from Claim \ref{cla:containZTZ5} imply the following: if 
	  \begin{align}\label{eq:conditionalRest5}
		  \epsilon R (1 + \epsilon)^2 r_T(z) \leq \frac12 \, , 
	  \end{align} 
	  then 
	  \[r_{Z T Z^*}(z) \leq 2 (1 + \epsilon)^2 r_T(z) \, . \]
	  
	  Under \eqref{eq:conditionalRest5}, then, we deduce for $z \in \sigma_\epsilon(Z T Z^*)$ that 
	  \[r_{T}(z) \geq \frac{1}{2 (1 + \epsilon)^2} r_{ZTZ^*}(z)\geq \frac{1}{2 \epsilon (1 + \epsilon)^2} \, , \]
	  the RHS of which is $\geq \delta_0^{-1}$ if $\delta_0 \geq 2 \epsilon (1 + \epsilon)^2$, as we have from \eqref{eq:condDelta05}.
	  
	  On the other hand, if \eqref{eq:conditionalRest5} is false then 
	  \[r_T(z) > \frac{1}{\epsilon} \cdot \frac{1}{2 R (1 + \epsilon)^2} \geq \frac{1}{\delta_0} \]
	  by \eqref{eq:condDelta05}, hence in this case as well we have that $z \in \sigma_{\delta_0}(T)$. 
	  \end{proof}  }

\subsection{Approximate Singular Value Decomposition}\label{subsec:SVD5}

 Proposition \ref{prop:approxSchurToHouseholder5} above reduces the estimation of pseudospectrum of an approximately-known $M \in M_{n \times n}(\C)$ to estimation of resolvent norms of an upper triangular matrix $T$ which, in practice, will have floating-point entries. 
	In view of \eqref{eq:lastSingularValue5}, we seek estimates of the smallest singular value of $z - T, z \in \C$. For this we will use the following quantitative ``approximate'' Singular Value Decomposition due to Rump.

\begin{proposition}[\cite{rump2010verified}, Theorem 3.1]\label{thm:rumpsvd}
	Let $B, U, V \in M_{n \times n}(\C)$ be given and define $E_U = I - U^* U, E_V = I - V^* V$. Let $\alpha, \beta \in (0,1)$ and assume that 
	\[
		\|E_U\| \leq \alpha < 1, \qquad \|E_V\| \leq \beta < 1.
	\]
	Define $\Sigma = U^* B V$ and let $\Sigma = D + E_\Sigma$ where $D$ is diagonal and $E_\Sigma$ has diagonal entries all equal to 0. 
	
	Then, there is a permutation $\nu: \{1, \ldots, n\} \to \{1, \ldots, n\}$ such that
	\[
		\frac{D_{ii} - \|E_\Sigma\|}{\sqrt{1+\alpha}\sqrt{1+\beta}} \leq \sigma_{\nu(i)}(B) \leq \frac{D_{ii} + \|E_\Sigma\|}{\sqrt{1-\alpha}\sqrt{1-\beta}}
	\]
	for all $1 \leq i \leq n$. 
\end{proposition}

\begin{corollary}
	Under the assumptions of the previous theorem, it holds that 
	\[\|B^{-1}\| \leq \frac{1}{\theta} \, , \]
	provided 
	\[
		\theta := \min_{1 \leq i \leq n} \frac{D_{ii} - \|E_\Sigma\|_2}{\sqrt{1+\alpha}\sqrt{1+\beta}}
	\]
	is positive. 
\end{corollary}

{Equipped with a means of estimating singular values of approximately-known matrices, we now turn to address issue of bounding resolvent norms along infinite subsets, with an eye to simple closed curves $\gamma \subset \C$ as needed in Proposition \ref{prop:approxSchurToHouseholder5}. The following is a straightforward consequence of the first Resolvent Identity. }

\begin{lemma}\label{lem:meshingSVD5}
	Let $\gamma \subset \C$ be any bounded set, and let $\xi \subset \gamma$. Let $B\in M_{n \times n}$. Assume \begin{itemize}
		\item[(i)] $r_B(z) \leq C_\xi$ for some $C_\xi > 0$ and all $z \in \xi$; and 
		\item[(ii)] $\xi$ is $\eta$-dense\footnote{That is, for all $z \in \gamma$ there exists $z' \in \xi$ such that $|z - z'| \leq \eta$.} in $\gamma$ for some $\eta < \frac{1}{C_\xi}$.
	\end{itemize}
	Then, 
	\[r_B(z) \leq \frac{C_\xi}{1 - \eta C_\xi} \quad \text{ for all } z \in \gamma \, .\]
\end{lemma}

\subsubsection*{Low-cost estimates of matrix norms}

The estimates presented thus far invoke error terms in terms of  the spectral norm $\| \cdot \|$. This is potentially bad, since in practice the singular value decomposition itself is needed to get good upper bounds on spectral norms. On the other hand, the error-type matrices  $E_Z, E_M$ in Proposition \ref{prop:approxSchur5} or $E_\Sigma, E_U, E_V$ in Proposition \ref{thm:rumpsvd} are known to be small, and so we can get effectively estimate their spectral norms using weaker yet compu
tationally efficient upper bounds for spectral norms. The following is particularly convenient for our purposes.

Below, for a square matrix $B \in M_{n \times n}(\C)$ we write $|B|$ for the $n\times n$ matrix with entries 
\[|B|_{ij} = |B_{ij}| \,. \]
\begin{lemma}[{\cite[Section 3]{rump2010verified}}]\label{lemma:normPerron}
	Let $B \in M_{n \times n}(\C)$ and let $v \in \R^n$ be an arbitrary vector with strictly positive entries. Then, 
	\[
		\| B \| \leq \sqrt{\max_i \frac{(|B^T| |B| v)_i}{v_i}} \, .
	\]
\end{lemma}

\subsection{Implementation}\label{subsec:implementation5}

An essential tool in our approach is the use of \textbf{Self-Validated Numerical Methods}: computational techniques whose outputs have the same strength as formal mathematical proofs. For an in-depth exposition, we refer the interested reader to the foundational text \cite{Tucker2011}. 

The central idea is to implement set-valued arithmetic that satisfies the \textbf{inclusion principle}: given a function \( f \) and a set \( I \) representable in floating-point arithmetic, the method returns a computable set \( J \) such that
\[
f(I) \subseteq J.
\]
A canonical example is \textbf{Interval Arithmetic}, which is available in the Julia package \cite{IntervalArithmetic.jl}. We use this package for tasks such as the validated discretization of operators, see \cite[Remark 3.18]{NisoliCrush2023} for details on the validated FFT method. However, while interval arithmetic excels in pointwise operations, its use for linear algebra is computationally expensive and often impractical for large-scale matrix operations.

To address this, we adopt the \textbf{ball arithmetic framework} introduced by Rump \cite{Rump1999}, which strikes a balance between enclosure sharpness, computational speed, and rigor. We implemented this approach in the Julia package \cite{BallArithmetic.jl}, which defines a \texttt{BallMatrix} type composed of a matrix of centers \( A_c \) and a matrix of radii \( A_r \), representing the set:
\[
\hat{A} = \left\{ A \in \mathbb{R}^{n \times m} \,\middle|\, |A_{ij} - (A_c)_{ij}| \leq (A_r)_{ij} \right\}.
\]
Given two such ball matrices \( \hat{A} \) and \( \hat{B} \), the algorithm \texttt{MMul4} provides a rigorous enclosure \( \hat{C} \) for the product \( \hat{A} \hat{B} \), ensuring
\[
A \cdot B \in \hat{C} \quad \text{for all } A \in \hat{A},\, B \in \hat{B}.
\]

This implementation relies on optimized BLAS routines and achieves high performance by compiling a custom version of OpenBLAS with the \texttt{ConsistentFPCSR=1} flag, which enforces consistent rounding modes across threads. This configuration is handled automatically during installation of \texttt{BallArithmetic.jl}, which also performs checks to verify floating-point consistency.

The framework developed by Rump and collaborators is designed to certify numerical computations. Given a \texttt{BallMatrix} \( \hat{M} \) that encloses a discretized operator, we numerically compute the Schur decomposition of its center \( M_c \), and then obtain rigorous enclosures for the error terms \( E_Z \) and \( E_M \) from Proposition~\ref{prop:approxSchurToHouseholder5}, uniformly over all matrices \( M \in \hat{M} \).

This is achieved by evaluating the expression
\[
\hat{M} - \widehat{Z^* T Z},
\]
where \( \widehat{Z^* T Z} \) is a rigorously computed enclosure of the matrix product \( Z^* T Z \), accounting for all rounding errors. The resulting difference yields a validated enclosure for the error in the approximate Schur factorization over the entire ball \( \hat{M} \).

Similarly, we treat the error terms \( E_U, E_V, E_{\Sigma} \) in Proposition~\ref{thm:rumpsvd}, enabling a fully validated computation of the singular values of the matrix \( \tilde{z} - T \), for a fixed, representable complex floating-point number \( \tilde{z} \).

Let \( \hat{\sigma}_n(\tilde{z}) = [\sigma_{\min}, \sigma_{\max}] \) be a rigorous enclosure for the smallest singular value of \( \tilde{z} - T \). Then, by the first resolvent inequality, we obtain the bound
\[
\sup_{z \in B(\tilde{z}, \tfrac{\sigma_{\min}}{2})} r(z, T) \leq \frac{2}{\sigma_{\min}},
\]
where \( r(z, T) = \| (z - T)^{-1} \| \) denotes the norm of the resolvent. This provides a validated control on the resolvent norm in a neighborhood of \( \tilde{z} \), which in turn allows us to rigorously cover curves in the complex plane.

To extend this to an entire contour \( C \subset \mathbb{C} \), we implement an adaptive algorithm based on certified enclosures of the smallest singular value. The contour is initially partitioned into arcs \( \gamma_i \), each represented by endpoints \( (\tilde{z}_{i,1}, \tilde{z}_{i,2}) \in \gamma_i \). For each arc, we compute a rigorous enclosure \( \hat{\sigma}(\tilde{z}_{i,1}) = [\sigma_{\min}^{(i)}, \sigma_{\max}^{(i)}] \).

To ensure validity of the resolvent estimate over the arc, we require
\[
\mathrm{diam}(\gamma_i) \leq \frac{\sigma_{\min}^{(i)}}{2},
\]
so that the resolvent norm satisfies
\[
\| (z - T)^{-1} \| \leq \frac{2}{\sigma_{\min}^{(i)}} \quad \text{for all } z \in \gamma_i.
\]
Arcs that violate this condition are recursively bisected and added to a work queue. The algorithm is naturally parallelizable, as the singular value computations are independent and constitute the main computational cost.

Since the number of arcs is inversely proportional to the local bound \( \sigma_{\min}^{(i)} \), and \( r(z, T) \approx 1/\sigma_{\min}(z) \), the total number of SVD evaluations is proportional to the integral
\[
\int_C r(z, T)\, |dz|,
\]
which can be estimated numerically and reflects the local analytical complexity of the problem, and serves as a useful heuristic to estimate computational cost and runtime.

%% file: outlook.tex
\label{sec:outlook}

We close this paper with some comments on our perception of the role of pseudospectrum in the general problem of eigenvalue estimation. At the end, we propose some future research directions. 

\subsection{Pseudospectrum and the modulus of continuity of spectrum}

    A natural idea is to tackle the enclosure and exclosure of eigenvalues of an operator $A$ simultaneously by evaluating the modulus of continuity of points in $\sigma(A)$ under perturbation, with the hope that a sufficiently fine discretization ensures that the spectrum of the discretized operator is close to that of the original operator.
    This approach applies in cases where the discretization is norm-continuous with respect to a single norm (as in, e.g., \cite{froyland1997computer}), but also in cases where only a DFLY inequality, as in Assumption \ref{ass:DFLY} following \cite{keller1999stability}, is available.
    
    On the other hand, in some sense pseudospectrum is unavoidable in estimates of moduli of continuity of eigenvalues. 
    Even for finite-dimensional matrices, the \emph{worst-case modulus of continuity of the spectrum is dictated by the shape of the pseudospectrum}, as indicated by \edited{the alternative characterization of pseudospectrum given in Theorem \ref{thm:modulusCtyPseudospec1}, c.f. \cite[Theorem 2.1]{trefethen1999pseudospectra}.}
    This idea is implicit, for instance, on a close reading of \cite{keller1999stability}, which requires taking keyhole contour integrals close to the spectrum of the unperturbed operator. 
    
    Moduli of continuity of spectrum therefore rely on pointwise evaluations of resolvent norms very close to the spectrum -- these estimates necessarily involve very large numbers, potentially leading to dramatic computational error. In contrast, our approach to exclosure allows for the estimation of $\delta$-pseudospectra for non-perturbatively small $\delta$, and requires, therefore, estimation of smaller resolvent norms. Once exclosure is achieved, eigenvalue estimates can then be refined using standard `enclosure' methods, e.g., Newton-Kantorovich as in \cite[Section 9.3.1]{nakao2019numerical}. 
    
    That being said, in some cases one can extract more direct information about stability of spectrum under special perturbations. This is relevant, for example, to semigroups generated by relatively compact perturbations of operators with exactly known spectra, such as the Laplacian (c.f. \cite{james2017fourier}), for which one can cast discretization perturbations in a special form leaving unchanged the large-wavenumber behavior of the operator. However, such settings are restrictive; we are aware, for instance, of no special structure for transfer operators of expanding maps in `general position'.

\subsection{Pointwise bounds on resolvent norms}

    The scheme outlined in \cite[Section 9.3.2]{nakao2019numerical} provides a method for ruling out spectrum near a given value $\mu \in \C$ for an operator $A$ on a Hilbert space $H$.
    The key idea is to verify that $(A-\mu)^* (A-\mu) \geq \gamma$ for some real number $\gamma > 0$, where ``$\geq$'' is the usual partial order on self-adjoint operators. This can be achieved using techniques from self-adjoint spectral theory, such as those in \cite[Chapter 10]{nakao2019numerical}. From this, one can conclude that no eigenvalues for $A$ exist in the ball $B_{\sqrt{\gamma}}(\mu) \subset \C$ (see \cite[Theorem 9.18]{nakao2019numerical} for details).
    
    This approach essentially reduces to estimating the smallest singular value of $A-\mu$, or equivalently, to an upper bound on the resolvent norm of $A$ at $\mu$. 
    However, applying this method to exclosure requires sampling across possibly large domains in $\C$ where spectrum may occur. In contrast, our approach leverages the subharmonicity of the resolvent norm (Lemma \ref{lem:subharmonic2}) and only requires evaluating resolvent norms along curves in $\C$, i.e., 
    we need to bound the resolvent norm on $1$ dimensional objects instead than on an open subset of the plane, which leads to much smaller computational costs. 
    An interesting direction for future work would be to conduct case studies comparing the performance and computational cost of the `pointwise' method in \cite[Section 9.3.2]{nakao2019numerical} with the `curve-wise' method presented in Section \ref{sec:householder}.

\subsection{Future directions for transfer operators}

{

The framework we use -- that of DFLY inequalities and strong-weak pairs of Banach spaces -- is standard in modern spectral approaches to the study of chaotic dynamical systems, with applications including Anosov / Axiom A diffeomorphisms and flows (see, e.g., \cite{gouezel2006banach, demers2008liverani, Baladi_2016, giulietti2013, butterley2021locating}) -- as well as singular dynamics such as dispersing billiards \cite{baladi2018exponential}. 

The pseudospectral approach put forward in this paper is applicable to any of these frameworks, and in essence converts the problem of rigorously counting and enclosing resonances into a computational one. Implementing this program practically generates a novel challenge: the discovery of DFLY frameworks amenable to computation, e.g., utilizing (where possible) weighted $\ell^1$ spaces, as we have done in Section \ref{sec:analyticframework} for analytic uniformly expanding maps.

}

\subsection{Some additional future directions}

    The authors hope that the method presented here opens up new possibilities for studying spectral theory of non-selfadjoint operators using computer-assisted proof. 
    We feel that there still remain some interesting open avenues for further refinement and extension of the method. Possibilities include: 
    \begin{itemize}
        \item[(a)] Extension to spectrum of unbounded operators, e.g., to differential operators appearing in PDE, the idea being that in this setting it is likely much easier computationally to work at level of (usually unbounded) generators than to work with time-integrated semigroups; 
        \item[(b)] Proper analysis of the computational cost of the method in this paper versus other methods, e.g., pointwise bounds on resolvent norms; 
        \item[(c)] \edited{Comparison of the pseudospectral enclosures obtained for different truncation schemes (i.e., choices of $\tilde{A}$); and}
        \item[(d)] Improvements to the technical implementation, e.g., automating the discovery of curves $\gamma$ isolating eigenvalues as in Lemma \ref{lem:subharmonic2} or adaptive step-sizing when deriving a discrete mesh for these curves. 
    \end{itemize}

%% file: biblio.bib
@article{budisic2012applied,
  title={Applied koopmanism},
  author={Budi{\v{s}}i{\'c}, Marko and Mohr, Ryan and Mezi{\'c}, Igor},
  journal={Chaos: An Interdisciplinary Journal of Nonlinear Science},
  volume={22},
  number={4},
  year={2012},
  publisher={American Institute of Physics}
}

@article{brunton2022modern,
  title={Modern Koopman Theory for Dynamical Systems},
  author={Brunton, Steven L. and Budi{\v{s}}i{\'c}, Marko and Kaiser, Eurika and Kutz, J. Nathan},
  journal={SIAM Review},
  volume={64},
  number={2},
  pages={229--340},
  year={2022},
  publisher={SIAM}
}

@article{das2020koopman,
  title={Koopman spectra in reproducing kernel Hilbert spaces},
  author={Das, Suddhasattwa and Giannakis, Dimitrios},
  journal={Applied and Computational Harmonic Analysis},
  volume={49},
  number={2},
  pages={573--607},
  year={2020},
  publisher={Elsevier}
}

@article{colbrook2024rigorous,
  title={Rigorous data-driven computation of spectral properties of Koopman operators for dynamical systems},
  author={Colbrook, Matthew J. and Townsend, Alex},
  journal={Communications on Pure and Applied Mathematics},
  year={2024},
  publisher={Wiley}
}

@article{williams2015data,
  title={A data--driven approximation of the Koopman operator: Extending dynamic mode decomposition},
  author={Williams, Matthew O. and Kevrekidis, Ioannis G. and Rowley, Clarence W.},
  journal={Journal of Nonlinear Science},
  volume={25},
  number={6},
  pages={1307--1346},
  year={2015},
  publisher={Springer}
}

@article{hansen2011solvability,
  title={On the solvability complexity index, the $n$-pseudospectrum and approximations of spectra of operators},
  author={Hansen, Anders},
  journal={Journal of the American Mathematical Society},
  volume={24},
  number={1},
  pages={81--124},
  year={2011}
}

@article{colbrook2022foundations,
  title={The foundations of spectral computations via the solvability complexity index hierarchy},
  author={Colbrook, Matthew J and Hansen, Anders C},
  journal={Journal of the European Mathematical Society},
  volume={25},
  number={12},
  pages={4639--4718},
  year={2022}
}

@phdthesis{colbrook2020foundations,
  title={The foundations of infinite-dimensional spectral computations},
  author={Colbrook, Matthew},
  year={2020}
}

@article{ben2015new,
  title={New barriers in complexity theory: On the solvability complexity index and the towers of algorithms},
  author={Ben-Artzi, Jonathan and Hansen, Anders C and Nevanlinna, Olavi and Seidel, Markus},
  journal={Comptes Rendus Mathematique},
  volume={353},
  number={10},
  pages={931--936},
  year={2015},
  publisher={Elsevier}
}

@article{klus2018data,
  title={Data-driven model reduction and transfer operator approximation},
  author={Klus, Stefan and N{\"u}ske, Feliks and Koltai, P{\'e}ter and Wu, Hao and Kevrekidis, Ioannis G. and Sch{\"u}tte, Christof and No{\'e}, Frank},
  journal={Journal of Nonlinear Science},
  volume={28},
  number={3},
  pages={985--1010},
  year={2018},
  publisher={Springer}
}

@article{froyland2009almost,
  title={Almost-invariant sets and invariant manifolds---connecting probabilistic and geometric descriptions of coherent structures in flows},
  author={Froyland, Gary and Padberg, Kathrin},
  journal={Physica D: Nonlinear Phenomena},
  volume={238},
  number={16},
  pages={1507--1523},
  year={2009},
  publisher={Elsevier}
}

@article{akindji2026convergence,
  title={Convergence properties of dynamic mode decomposition for analytic interval maps},
  author={Akindji, Elliz and Slipantschuk, Julia and Bandtlow, Oscar F and Just, Wolfram},
  journal={Communications on Pure and Applied Mathematics},
  volume={79},
  number={2},
  pages={179--206},
  year={2026},
  publisher={Wiley Online Library}
}

@incollection{jenkinson2024dynamical,
  title={A dynamical approach to validated numerics},
  author={Jenkinson, Oliver and Pollicott, Mark},
  booktitle={A Vision for Dynamics in the 21st Century: The Legacy of Anatole Katok},
  pages={168},
  year={2024},
  publisher={Cambridge University Press}
}

@article{pujals2006expanding,
  title={{Expanding maps of the circle rerevisited: positive Lyapunov exponents in a rich family}},
  author={Pujals, Enrique R and Robert, Leonel and Shub, Michael},
  journal={Ergodic Theory and Dynamical Systems},
  volume={26},
  number={6},
  pages={1931--1937},
  year={2006},
  publisher={Cambridge University Press}
}

@article{tischler2000blaschke,
  title={{Blaschke products and expanding maps of the circle}},
  author={Tischler, David},
  journal={Proceedings of the American Mathematical Society},
  volume={128},
  number={2},
  pages={621--622},
  year={2000}
}

@article{martin1983finite,
  title={{On finite Blaschke products whose restrictions to the unit circle are exact endomorphisms}},
  author={Martin, NFG},
  journal={Bulletin of the London Mathematical Society},
  volume={15},
  number={4},
  pages={343--348},
  year={1983},
  publisher={Wiley Online Library}
}

@article{bauer1960norms,
  title={Norms and exclusion theorems},
  author={Bauer, Friedrich L and Fike, Charles T},
  journal={Numerische Mathematik},
  volume={2},
  pages={137--141},
  year={1960},
  publisher={Springer}
}

@article{froyland1997computer,
  title={Computer-assisted bounds for the rate of decay of correlations},
  author={Froyland, Gary},
  journal={Communications in Mathematical Physics},
  volume={189},
  pages={237--257},
  year={1997},
  publisher={Springer}
}

@article{globevnik1976norm,
  title={Norm-constant analytic functions and equivalent norms},
  author={Globevnik, Josip},
  journal={Illinois Journal of Mathematics},
  volume={20},
  number={3},
  pages={503--506},
  year={1976},
  publisher={Duke University Press}
}

@article{shargorodsky2008level,
  title={On the level sets of the resolvent norm of a linear operator},
  author={Shargorodsky, Eugene},
  journal={Bulletin of the London Mathematical Society},
  volume={40},
  number={3},
  pages={493--504},
  year={2008},
  publisher={Wiley Online Library}
}

@article{householder1958approximate,
  title={The approximate solution of matrix problems},
  author={Householder, Alston S},
  journal={Journal of the ACM (JACM)},
  volume={5},
  number={3},
  pages={205--243},
  year={1958},
  publisher={ACM New York, NY, USA}
}

@article{li1976finite,
  title={{Finite approximation for the Frobenius-Perron operator. A solution to Ulam's conjecture}},
  author={Li, Tien-Yien},
  journal={Journal of Approximation theory},
  volume={17},
  number={2},
  pages={177--186},
  year={1976},
  publisher={Elsevier}
}

@article{pollicott2000computing,
  title={Computing invariant densities and metric entropy},
  author={Pollicott, Mark and Jenkinson, Oliver},
  journal={Communications in Mathematical Physics},
  volume={211},
  number={3},
  pages={687--703},
  year={2000},
  publisher={Springer}
}

@article{jenkinson2018rigorous,
  title={Rigorous computation of diffusion coefficients for expanding maps},
  author={Jenkinson, O and Pollicott, M and Vytnova, Polina},
  journal={Journal of Statistical Physics},
  volume={170},
  pages={221--253},
  year={2018},
  publisher={Springer}
}

@article{jenkinson2005orthonormal,
  title={Orthonormal expansions of invariant densities for expanding maps},
  author={Jenkinson, Oliver and Pollicott, Mark},
  journal={Advances in Mathematics},
  volume={192},
  number={1},
  pages={1--34},
  year={2005},
  publisher={Elsevier}
}

@article{baladi1993spectra,
  title={On the spectra of randomly perturbed expanding maps},
  author={Baladi, Viviane and Young, L-S},
  journal={Communications in Mathematical Physics},
  volume={156},
  number={2},
  pages={355--385},
  year={1993},
  publisher={Springer}
}

@article{froyland2007ulam,
  title={On Ulam approximation of the isolated spectrum and eigenfunctions of hyperbolic maps},
  author={Froyland, Gary},
  journal={Discrete Contin. Dyn. Syst. Ser. A},
  volume={17},
  number={3},
  pages={203--221},
  year={2007}
}

@book{ulam1960collection,
  title={A collection of mathematical problems},
  author={Ulam, S. M.},
  series={Interscience Tracts in Pure and Applied Mathematics},
  number={8},
  publisher={Interscience Publishers},
  address={New York-London},
  year={1960}
}

@article{dellnitz1999approximation,
  title={On the approximation of complicated dynamical behavior},
  author={Dellnitz, Michael and Junge, Oliver},
  journal={SIAM Journal on Numerical Analysis},
  volume={36},
  number={2},
  pages={491--515},
  year={1999},
  publisher={SIAM}
}

@article{bandtlow2020lagrange,
  title={Lagrange approximation of transfer operators associated with holomorphic data},
  author={Bandtlow, Oscar F and Slipantschuk, Julia},
  journal={arXiv preprint arXiv:2004.03534},
  year={2020}
}

@book{ding2010statistical,
  title={Statistical properties of deterministic systems},
  author={Ding, Jiu and Zhou, Aihui},
  year={2010},
  publisher={Springer Science \& Business Media}
}

@book{demers2021transfer,
  title={Transfer operators in hyperbolic dynamics. An introduction},
  author={Demers, MF and Kiamari, N and Liverani, C},
  year={2021},
  publisher={IMPA}
}

@article{liverani2001rigorous,
  title={Rigorous numerical investigation of the statistical properties of piecewise expanding maps. A feasibility study},
  author={Liverani, Carlangelo},
  journal={Nonlinearity},
  volume={14},
  number={3},
  pages={463},
  year={2001},
  publisher={IOP Publishing}
}

@article{eckmann1985ergodic,
  title={Ergodic theory of chaos and strange attractors},
  author={Eckmann, J-P and Ruelle, David},
  journal={Reviews of modern physics},
  volume={57},
  number={3},
  pages={617},
  year={1985},
  publisher={APS}
}

@book{trefethenEmbree,
 ISBN = {9780691119465},
 URL = {http://www.jstor.org/stable/j.ctvzxx9kj},
 author = {Lloyd N. Trefethen and Mark Embree},
 publisher = {Princeton University Press},
 title = {Spectra and Pseudospectra: The Behavior of Nonnormal Matrices and Operators},
 urldate = {2023-04-10},
 year = {2005}
}

@article{keller1982stochastic,
  title={Stochastic stability in some chaotic dynamical systems},
  author={Keller, Gerhard},
  journal={Monatshefte f{\"u}r Mathematik},
  volume={94},
  pages={313--333},
  year={1982},
  publisher={Springer}
}

@article{faure2019ruelle,
  title={Ruelle spectrum of linear pseudo-Anosov maps},
  author={Faure, Fr{\'e}d{\'e}ric and Gou{\"e}zel, S{\'e}bastien and Lanneau, Erwan},
  journal={Journal de l’{\'E}cole polytechnique—Math{\'e}matiques},
  volume={6},
  pages={811--877},
  year={2019}
}

@article{artuso1990recycling,
  title={Recycling of strange sets: I. Cycle expansions},
  author={Artuso, Roberto and Aurell, Erik and Cvitanovic, Predrag},
  journal={Nonlinearity},
  volume={3},
  number={2},
  pages={325},
  year={1990},
  publisher={IOP Publishing}
}

@article{artuso1990recycling2,
  title={Recycling of strange sets: II. Applications},
  author={Artuso, Roberto and Aurell, Erik and Cvitanovic, Predrag},
  journal={Nonlinearity},
  volume={3},
  number={2},
  pages={361},
  year={1990},
  publisher={IOP Publishing}
}

@article{galatolo2023general,
  title={A general framework for the rigorous computation of invariant densities and the coarse-fine strategy},
  author={Galatolo, Stefano and Monge, Maurizio and Nisoli, Isaia and Poloni, Federico},
  journal={Chaos, Solitons \& Fractals},
  volume={170},
  pages={113329},
  year={2023},
  publisher={Elsevier}
}

@article{slipantschuk2022resonances,
  title={{Resonances for rational Anosov maps on the torus}},
  author={Slipantschuk, Julia and Bandtlow, Oscar F and Just, Wolfram},
  journal={arXiv preprint arXiv:2211.05925},
  year={2022}
}

@article{galatolo2020existence,
  title={Existence of noise induced order, a computer aided proof},
  author={Galatolo, Stefano and Monge, Maurizio and Nisoli, Isaia},
  journal={Nonlinearity},
  volume={33},
  number={9},
  pages={4237},
  year={2020},
  publisher={IOP Publishing}
}

@article{galatolo2015elementary2,
  title={An elementary way to rigorously estimate convergence to equilibriumand escape rates},
  author={Galatolo, Stefano and Nisoli, Isaia and Saussol, Beno{\^\i}t},
  journal={Journal of Computational Dynamics},
  volume={2},
  number={1},
  pages={51--64},
  year={2015},
  publisher={Journal of Computational Dynamics}
}

@article{sarihan2021quantitative,
  title={Quantitative spectral perturbation theory for compact operators on a Hilbert space},
  author={Sar{\i}han, Ay{\c{s}}e G{\"u}ven and Bandtlow, Oscar F},
  journal={Linear Algebra and its Applications},
  volume={610},
  pages={169--202},
  year={2021},
  publisher={Elsevier}
}

@article{bandtlow2008resolvent,
  title={Resolvent estimates for operators belonging to exponential classes},
  author={Bandtlow, Oscar F},
  journal={Integral Equations and Operator Theory},
  volume={61},
  pages={21--43},
  year={2008},
  publisher={Springer}
}

@article{henrici1962bounds,
  title={Bounds for iterates, inverses, spectral variation and fields of values of non-normal matrices},
  author={Henrici, Peter},
  journal={Numerische Mathematik},
  volume={4},
  pages={24--40},
  year={1962},
  publisher={Springer}
}

@article{slipantschuk2013analytic,
  title={Analytic expanding circle maps with explicit spectra},
  author={Slipantschuk, Julia and Bandtlow, Oscar F and Just, Wolfram},
  journal={Nonlinearity},
  volume={26},
  number={12},
  pages={3231},
  year={2013},
  publisher={IOP Publishing}
}

@inproceedings{bandtlow2017spectral,
  title={Spectral structure of transfer operators for expanding circle maps},
  author={Bandtlow, Oscar F and Just, Wolfram and Slipantschuk, Julia},
  booktitle={Annales de l'Institut Henri Poincar{\'e} C, Analyse non lin{\'e}aire},
  volume={34},
  number={1},
  pages={31--43},
  year={2017},
  organization={Elsevier}
}

@article{james2017fourier,
  title={Fourier--taylor approximation of unstable manifolds for compact maps: Numerical implementation and computer-assisted error bounds},
  author={James, JD Mireles},
  journal={Foundations of Computational Mathematics},
  volume={17},
  pages={1467--1523},
  year={2017},
  publisher={Springer}
}

@article{martine2022microscopic,
  title={Microscopic patterns in the 2D phase-field-crystal model},
  author={Martine-La Boissoniere, Gabriel and Choksi, Rustum and Lessard, Jean-Philippe},
  journal={Nonlinearity},
  volume={35},
  number={3},
  pages={1500},
  year={2022},
  publisher={IOP Publishing}
}

@article{van2024toward,
  title={Toward Computational Morse--Floer Homology: Forcing Results for Connecting Orbits by Computing Relative Indices of Critical Points},
  author={van den Berg, Jan Bouwe and Gameiro, Marcio and Lessard, Jean-Philippe and Van der Vorst, Rob},
  journal={Foundations of Computational Mathematics},
  volume={24},
  number={5},
  pages={1739--1776},
  year={2024},
  publisher={Springer}
}

@article{arioli2019non,
  title={Non-radial solutions for some semilinear elliptic equations on the disk},
  author={Arioli, Gianni and Koch, Hans},
  journal={Nonlinear Analysis},
  volume={179},
  pages={294--308},
  year={2019},
  publisher={Elsevier}
}

@article{arioli2010integration,
  title={Integration of dissipative partial differential equations: a case study},
  author={Arioli, Gianni and Koch, Hans},
  journal={SIAM Journal on Applied Dynamical Systems},
  volume={9},
  number={3},
  pages={1119--1133},
  year={2010},
  publisher={SIAM}
}

@book{baladi2018dynamical,
  title={Dynamical zeta functions and dynamical determinants for hyperbolic maps},
  author={Baladi, Viviane},
  year={2018},
  publisher={Springer}
}

@article{de2016connecting,
  title={Connecting orbits for compact infinite dimensional maps: computer assisted proofs of existence},
  author={de la Llave, Rafael and Mireles James, Jason D},
  journal={SIAM Journal on Applied Dynamical Systems},
  volume={15},
  number={2},
  pages={1268--1323},
  year={2016},
  publisher={SIAM}
}

@article{frommer2021pseudospectrum,
  title={Pseudospectrum enclosures by discretization},
  author={Frommer, Andreas and Jacob, Birgit and Vorberg, Lukas and Wyss, Christian and Zwaan, Ian},
  journal={Integral Equations and Operator Theory},
  volume={93},
  pages={1--31},
  year={2021},
  publisher={Springer}
}

@article{bogli2018local,
  title={Local convergence of spectra and pseudospectra},
  author={B{\"o}gli, Sabine},
  journal={Journal of Spectral Theory},
  volume={8},
  number={3},
  pages={1051--1098},
  year={2018}
}

@article{bogli2014remarks,
  title={Remarks on the convergence of pseudospectra},
  author={B{\"o}gli, Sabine and Siegl, Petr},
  journal={Integral equations and operator theory},
  volume={80},
  pages={303--321},
  year={2014},
  publisher={Springer}
}

@article{davies2004spectral,
  title={Spectral pollution},
  author={Davies, Edward B and Plum, Michael},
  journal={IMA journal of numerical analysis},
  volume={24},
  number={3},
  pages={417--438},
  year={2004},
  publisher={Oxford University Press}
}

@article{wolff2001discrete,
  title={Discrete approximation of unbounded operators and approximation of their spectra},
  author={Wolff, MPH},
  journal={Journal of Approximation Theory},
  volume={113},
  number={2},
  pages={229--244},
  year={2001},
  publisher={Elsevier}
}

@inproceedings{wolff1998approximation,
  title={On the approximation of operators and the convergence of the spectra of the approximants},
  author={Wolff, MPH},
  booktitle={Recent Progress in Operator Theory: International Workshop on Operator Theory and Applications, IWOTA 95, in Regensburg, July 31--August 4, 1995},
  pages={279--283},
  year={1998},
  organization={Springer}
}

@article{wormell2019spectral,
  title={Spectral Galerkin methods for transfer operators in uniformly expanding dynamics},
  author={Wormell, Caroline},
  journal={Numerische Mathematik},
  volume={142},
  pages={421--463},
  year={2019},
  publisher={Springer}
}

@article{bahsoun2018rigorous,
  title={A rigorous computational approach to linear response},
  author={Bahsoun, Wael and Galatolo, Stefano and Nisoli, Isaia and Niu, Xiaolong},
  journal={Nonlinearity},
  volume={31},
  number={3},
  pages={1073},
  year={2018},
  publisher={IOP Publishing}
}

@article{bahsoun2006rigorous,
  title={Rigorous numerical approximation of escape rates},
  author={Bahsoun, Wael},
  journal={Nonlinearity},
  volume={19},
  number={11},
  pages={2529},
  year={2006},
  publisher={IOP Publishing}
}

@article{bahsoun2011invariant,
  title={Invariant densities and escape rates: rigorous and computable approximations in the $L^\infty$-norm},
  author={Bahsoun, Wael and Bose, Christopher},
  journal={Nonlinear Analysis: Theory, Methods \& Applications},
  volume={74},
  number={13},
  pages={4481--4495},
  year={2011},
  publisher={Elsevier}
}

@article{korepanov2019explicit,
  title={Explicit coupling argument for non-uniformly hyperbolic transformations},
  author={Korepanov, Alexey and Kosloff, Zemer and Melbourne, Ian},
  journal={Proceedings of the Royal Society of Edinburgh Section A: Mathematics},
  volume={149},
  number={1},
  pages={101--130},
  year={2019},
  publisher={Royal Society of Edinburgh Scotland Foundation}
}

@article{keller1999stability,
  title={Stability of the spectrum for transfer operators},
  author={Keller, Gerhard and Liverani, Carlangelo},
  journal={Annali della Scuola Normale Superiore di Pisa-Classe di Scienze},
  volume={28},
  number={1},
  pages={141--152},
  year={1999}
}

@article{gundlach2003sharp,
  title={{A sharp formula for the essential spectral radius of the Ruelle transfer operator on smooth and H{\"o}lder spaces}},
  author={Gundlach, Volker M and Latushkin, Yuri},
  journal={Ergodic Theory and Dynamical Systems},
  volume={23},
  number={1},
  pages={175--191},
  year={2003},
  publisher={Cambridge University Press}
}

@article{galatolo2014elementary,
  title={An elementary approach to rigorous approximation of invariant measures},
  author={Galatolo, Stefano and Nisoli, Isaia},
  journal={SIAM Journal on Applied Dynamical Systems},
  volume={13},
  number={2},
  pages={958--985},
  year={2014},
  publisher={SIAM}
}

@article{baladi1999approximation,
  title={Approximation of nonessential spectrum of transfer operators},
  author={Baladi, Viviane and Holschneider, Matthias},
  journal={Nonlinearity},
  volume={12},
  number={3},
  pages={525},
  year={1999},
  publisher={IOP Publishing}
}

@book{baladi2000positive,
  title={Positive transfer operators and decay of correlations},
  author={Baladi, Viviane},
  volume={16},
  year={2000},
  publisher={World scientific}
}

@article{liverani1995decay,
  title={Decay of correlations},
  author={Liverani, Carlangelo},
  journal={Annals of Mathematics},
  volume={142},
  number={2},
  pages={239--301},
  year={1995},
  publisher={JSTOR}
}

@book{conway1973functions,
  title={Functions of one complex variable I},
  author={Conway, JB},
  journal={Graduate Texts in Mathematics},
  number={11},
  year={1978},
  publisher={Springer New York, NY},
  edition = {2},
}

@incollection{mayer1991continued,
  title={Continued fractions and related transformations},
  author={Mayer, Dieter H},
  booktitle={Ergodic Theory, Symbolic Dynamics and Hyperbolic Spaces},
  editor={Bedford, Tim and Keane, Michael and Series, Caroline},
  year={1991},
  publisher={Oxford University Press},
  address={Oxford},
  pages={175--222}
}

@article{froyland2014detecting,
  title={{Detecting isolated spectrum of transfer and Koopman operators with Fourier analytic tools}},
  author={Froyland, Gary and Gonz{\'a}lez-Tokman, Cecilia and Quas, Anthony},
  journal={Journal of Computational Dynamics},
  volume={1},
  number={2},
  pages={249--278},
  year={2014},
  publisher={Journal of Computational Dynamics}
}

@article{doeblin1937chaines,
  title={Sur des cha{\^\i}nes {\`a} liaisons compl{\`e}tes},
  author={Doeblin, Wolfgang and Fortet, Robert},
  journal={Bulletin de la Soci{\'e}t{\'e} Math{\'e}matique de France},
  volume={65},
  pages={132--148},
  year={1937}
}

@article{nussbaum1970radius,
  title={The radius of the essential spectrum},
  author={Nussbaum, Roger D},
  journal={Duke Mathematical Journal},
  volume={37},
  number={3},
  pages={473--478},
  year={1970},
  publisher={Duke University Press}
}

@article{hennion1993theoreme,
  title={Sur un th{\'e}or{\`e}me spectral et son application aux noyaux lipchitziens},
  author={Hennion, Hubert},
  journal={Proceedings of the American Mathematical Society},
  volume={118},
  number={2},
  pages={627--634},
  year={1993}
}

@article{marinescu1950theorie,
  title={Th{\'e}orie ergodique pour des classes d'op{\'e}rations non compl{\`e}tement continues},
  author={Marinescu, G},
  journal={Annals of Mathematics},
  volume={52},
  number={1},
  pages={140--147},
  year={1950},
  publisher={JSTOR}
}

@article{lasota2004existence,
  title={On the existence of invariant measures for piecewise monotonic transformations},
  author={Lasota, A and Yorke, James A},
  journal={Transactions of the American Mathematical Society},
  volume={186},
  pages={481--488},
  year={1973}
}

@article{butterley2021locating,
  title={Locating Ruelle--Pollicott resonances},
  author={Butterley, Oliver and Kiamari, Niloofar and Liverani, Carlangelo},
  journal={Nonlinearity},
  volume={35},
  number={1},
  pages={513},
  year={2021},
  publisher={IOP Publishing}
}

@article{baladi2018exponential,
  title={Exponential decay of correlations for finite horizon Sinai billiard flows},
  author={Baladi, Viviane and Demers, Mark F and Liverani, Carlangelo},
  journal={Inventiones mathematicae},
  volume={211},
  number={1},
  pages={39--177},
  year={2018},
  publisher={Springer}
}

@article{rychlik1989regularity,
  title={Regularity of the metric entropy for expanding maps},
  author={Rychlik, Marek},
  journal={Transactions of the American Mathematical Society},
  volume={315},
  number={2},
  pages={833--847},
  year={1989}
}

@article{liverani1995decayJSP,
  title={Decay of correlations for piecewise expanding maps},
  author={Liverani, Carlangelo},
  journal={Journal of Statistical Physics},
  volume={78},
  pages={1111--1129},
  year={1995},
  publisher={Springer}
}

@article{dyatlov2015stochastic,
  title={Stochastic stability of Pollicott--Ruelle resonances},
  author={Dyatlov, Semyon and Zworski, Maciej},
  journal={Nonlinearity},
  volume={28},
  number={10},
  pages={3511},
  year={2015},
  publisher={IOP Publishing}
}

@article{slipantschuk2017complete,
  title={Complete spectral data for analytic Anosov maps of the torus},
  author={Slipantschuk, Julia and Bandtlow, Oscar F and Just, Wolfram},
  journal={Nonlinearity},
  volume={30},
  number={7},
  pages={2667},
  year={2017},
  publisher={IOP Publishing}
}

@article{bahsoun2016rigorous,
  title={Rigorous approximation of diffusion coefficients for expanding maps},
  author={Bahsoun, Wael and Galatolo, Stefano and Nisoli, Isaia and Niu, Xiaolong},
  journal={Journal of Statistical Physics},
  volume={163},
  pages={1486--1503},
  year={2016},
  publisher={Springer}
}

@article{bressaud2002anosov,
  title={Anosov diffeomorphisms and coupling},
  author={Bressaud, Xavier and Liverani, Carlangelo},
  journal={Ergodic Theory and Dynamical Systems},
  volume={22},
  number={1},
  pages={129--152},
  year={2002},
  publisher={Cambridge University Press}
}

@article{isola1988resonances,
  title={Resonances in chaotic dynamics},
  author={Isola, Stefano},
  journal={Communications in mathematical physics},
  volume={116},
  pages={343--352},
  year={1988},
  publisher={Springer}
}

@article{lanford1982computer,
  title={A computer-assisted proof of the Feigenbaum conjectures},
  author={Lanford III, Oscar E},
  journal={Bulletin of the American Mathematical Society},
  volume={6},
  number={3},
  pages={427--434},
  year={1982}
}

@book{nakao2019numerical,
  title={Numerical verification methods and computer-assisted proofs for partial differential equations},
  author={Nakao, Mitsuhiro T and Plum, Michael and Watanabe, Yoshitaka},
  year={2019},
  publisher={Springer}
}

@article{lui1997continuation,
  title={Computation of Pseudospectra by continuation},
  author={Liu, S. H.},
  journal={SISC},
  volume={18},
  pages={565--573},
  year={1997},
  publisher={Unsure}
}

@Book{Tucker2011,
  author    = {Tucker, Warwick},
  publisher = {Princeton, NJ: Princeton University Press},
  title     = {Validated numerics. {A} short introduction to rigorous computations.},
  year      = {2011},
  isbn      = {978-0-691-14781-9; 978-1-400-83897-4},
  language  = {English},
  zbl       = {1231.65077},
  zbmath    = {5938010},
}

@misc{IntervalArithmetic.jl,
  author = {David P. Sanders and Luis Benet},
  title  = {IntervalArithmetic.jl},
  url    = {https://github.com/JuliaIntervals/IntervalArithmetic.jl},
  year   = {2014},
  doi    = {10.5281/zenodo.3336308}
}

@Article{Rump1999,
  author  = {Rump, Siegfried M.},
  journal = {BIT Numerical Mathematics},
  title   = {Fast and Parallel Interval Arithmetic},
  year    = {1999},
  pages   = {534--554},
  volume  = {39},
  doi     = {10.1023/A:1022374804152},
}

@misc{BallArithmetic.jl,
  author = {Ferranti, Luca and Nisoli, Isaia},
  title  = {{BallArithmetic}.jl},
  url    = {https://github.com/JuliaBallArithmetic/BallArithmetic.jl},
  year   = {2023},
}

@article{rump2010verified,
  title={Verified bounds for singular values, in particular for the spectral norm of a matrix and its inverse},
  author={Rump, S. M.},
  journal={BIT Numerical Mathematics},
  volume={51},
  pages={367--384},
  year={2011},
  publisher={Unsure}
}

@article{miyajima2014verified,
  title={Verified bounds for all the singular values of matrix},
  author={Miyajima, S.},
  journal={Japan J. Indust. Appl. Math.},
  volume={31},
  pages={513--539},
  year={2014},
  publisher={Unsure}
}

@article{florido2002schemes,
  title={Locating Pollicott-Ruelle resonances in chaotic dynamical systems: A class of numerical schemes.},
  author={Florido, R., Martin-Gonz\'alez J. M. and Gomez Lorente J. M.},
  journal={Physics review. E},
  volume={66},
  pages={046208},
  year={2002},
  publisher={Springer}
}

@article{venegeroles2008area,
  title={Leading Pollicot-Ruelle resonances for Chaotic Area-Preserving Maps.},
  author={Venegeroles, R.},
  journal={Physics review. E},
  volume={77},
  pages={027201},
  year={2008},
  publisher={Springer}
}

@article{horvat2009hybrid,
  title={A hybrid method for calculation of Ruelle–Pollicott resonances.},
  author={Horvat, M. and Veble, G.},
  journal={J. Phys. A: Math. Theor.},
  volume={45},
  pages={465101},
  year={2009},
  publisher={Springer}
}

@article{blum2000chaotic,
  title={Leading Ruelle resonances of chaotic maps},
  author={ Blum, G. and Agam, O.},
  journal={Phys. Rev. E},
  volume={62},
  pages={1977-82},
  year={2000},
  publisher={Springer}
}

@article{sano2002parametric,
  title={Parametric dependence of the Pollicott-Ruelle resonances for sawtooth maps},
  author={Sano, M.},
  journal={Phys. Rev. E},
  volume={66},
  pages={046211},
  year={2002},
  publisher={Springer}
}

@book{varga2004circles,
  title={Ger\u sgorin and His Circles},
  author={Richard S. Varga},
  isbn={978-3-540-21100-6},
    series = {Springer Series in Computational Mathematics},
    year = {2004},
    publisher = {Springer, Berlin, Heidelberg},
}

@book{trefethen1999pseudospectra,
  title={Computation of Pseudospectra},
  author={Trefethen, L. N.},
    year = {1999},
    publisher = {Acta Numerica},
}

@Article{NisoliCrush2023,
 Author = {Nisoli, Isaia and Taylor-Crush, Toby},
 Title = {Rigorous computation of linear response for intermittent maps},
 FJournal = {Journal of Statistical Physics},
 Journal = {J. Stat. Phys.},
 ISSN = {0022-4715},
 Volume = {190},
 Number = {12},
 Pages = {49},
 Note = {Id/No 192},
 Year = {2023},
 Language = {English},
 DOI = {10.1007/s10955-023-03174-8},
 zbMATH = {7782592},
 Zbl = {1536.37082}
}

@misc{DVN/Doubling,
author = {Blumenthal, Alex and Nisoli, Isaia and Taylor-Crush, Toby},
publisher = {Harvard Dataverse},
title = {{Eigenvalue enclosure logs for perturbed doubling map example}},
UNF = {UNF:6:Zvt7hEPPpOe5NUMARzZiMg==},
year = {2025},
version = {V1},
doi = {10.7910/DVN/AKKVZT},
url = {https://doi.org/10.7910/DVN/AKKVZT}
}

@misc{DVN/Blaschke,
author = {Blumenthal, Alex and Nisoli, Isaia and Taylor-Crush, Toby},
publisher = {Harvard Dataverse},
title = {{Eigenvalue enclosure logs for Blaschke product example}},
UNF = {UNF:6:eiXuCld7CTWQXLZ38q6vBg==},
year = {2025},
version = {V1},
doi = {10.7910/DVN/QFAIJL},
url = {https://doi.org/10.7910/DVN/QFAIJL}
}

@Article{Baladi_2016,
  author    = {Baladi, Viviane},
  journal   = {Journal of Statistical Physics},
  title     = {The Quest for the Ultimate Anisotropic Banach Space},
  year      = {2016},
  issn      = {1572-9613},
  month     = nov,
  number    = {3–4},
  pages     = {525--557},
  volume    = {166},
  doi       = {10.1007/s10955-016-1663-0},
  publisher = {Springer Science and Business Media LLC},
}

@Article{gouezel2006banach,
  author   = {Gou{\"e}zel, S{\'e}bastien and Liverani, Carlangelo},
  journal  = {Ergodic Theory Dyn. Syst.},
  title    = {Banach spaces adapted to {Anosov} systems},
  year     = {2006},
  issn     = {0143-3857},
  number   = {1},
  pages    = {189--217},
  volume   = {26},
  doi      = {10.1017/S0143385705000374},
  fjournal = {Ergodic Theory and Dynamical Systems},
  language = {English},
  zbl      = {1088.37010},
  zbmath   = {5014357},
}

@article{demers2008liverani,
 author = {Demers, Mark F. and Liverani, Carlangelo},
 title = {Stability of statistical properties in two-dimensional piecewise hyperbolic maps},
 fjournal = {Transactions of the American Mathematical Society},
 journal = {Trans. Am. Math. Soc.},
 issn = {0002-9947},
 volume = {360},
 number = {9},
 pages = {4777--4814},
 year = {2008},
 language = {English},
 doi = {10.1090/S0002-9947-08-04464-4},
 zbMATH = {5318542},
 Zbl = {1153.37019}
}

@article{giulietti2013,
 author = {Giulietti, Paolo and Liverani, Carlangelo and Pollicott, Mark},
 title = {Anosov flows and dynamical zeta functions},
 fjournal = {Annals of Mathematics. Second Series},
 journal = {Ann. Math. (2)},
 issn = {0003-486X},
 volume = {178},
 number = {2},
 pages = {687--773},
 year = {2013},
 language = {English},
 doi = {10.4007/annals.2013.178.2.6},
 zbMATH = {6203676},
 Zbl = {1418.37042}
}

@article{colbrook2024townsend,
    author    = {Colbrook, Matthew J. and Townsend, Alex},
    title     = {Rigorous data-driven computation of spectral properties
                 of {K}oopman operators for dynamical systems},
    journal   = {Comm. Pure Appl. Math.},
    volume    = {77},
    pages     = {221--283},
    year      = {2024},
    doi       = {10.1002/cpa.22125},
  }

@article{colbrook2019romanhansen,
    author    = {Colbrook, Matthew J. and Roman, Bogdan and Hansen, Anders C.},
    title     = {How to Compute Spectra with Error Control},
    journal   = {Phys. Rev. Lett.},
    volume    = {122},
    pages     = {250201},
    year      = {2019},
    doi       = {10.1103/PhysRevLett.122.250201},
  }

@misc{boulle2025colbrookconradie,                        
    author    = {Boull\'{e}, Nicolas and Colbrook, Matthew J. and Conradie, Gustav},
    title     = {Convergent Methods for {K}oopman Operators on
                 Reproducing Kernel {H}ilbert Spaces},                                                                                                                                                 
    year      = {2025},                                                                                                                                                                                
    eprint    = {2506.15782},                                                                                                                                                                          
    archivePrefix = {arXiv},                                                                                                                                                                           
  }

@misc{herwig2025slipantschuk,                                                                                                                                                                                        
    author    = {Herwig, April and Colbrook, Matthew J. and Junge, Oliver                                                                                                                              
                 and Koltai, P\'{e}ter and Slipantschuk, Julia},                                                                                                                                       
    title     = {Avoiding spectral pollution for transfer operators using residuals},                                                                                                                
    year      = {2025},                       
    eprint    = {2507.16915},                                                                                                                                                                          
    archivePrefix = {arXiv},
  }
